\documentclass[12pt,reqno,oneside]{amsproc}
\usepackage[english]{babel}
\usepackage{amssymb}
\usepackage{amsmath}
\usepackage{amsthm}
\usepackage{graphicx}
\usepackage{bbm}
\usepackage{mathtools}
\usepackage{tikz}
\usepackage{tikz-cd}
\usepackage{geometry}
\usepackage{xcolor}
\usepackage{enumitem}
\usepackage{microtype}
\usepackage{mathdots}	% provides \iddots
\usepackage{xifthen}
\allowdisplaybreaks
\usetikzlibrary{shapes.geometric,fit}

\definecolor{myurlcolor}{rgb}{0,0,0.4}
\definecolor{mycitecolor}{rgb}{0,0.5,0}
\definecolor{myrefcolor}{rgb}{0.5,0,0}
\usepackage[pagebackref]{hyperref}
\hypersetup{colorlinks,
linkcolor=myrefcolor,
citecolor=mycitecolor,
urlcolor=myurlcolor}

\usepackage[noabbrev,capitalise]{cleveref}

\numberwithin{equation}{section}

\theoremstyle{plain}
\newtheorem{dummy}{}[section]
\newtheorem{thm}[dummy]{Theorem}
\newtheorem{lemma}[dummy]{Lemma}
\newtheorem{prop}[dummy]{Proposition}
\newtheorem{cor}[dummy]{Corollary}

\newtheorem{defn}[dummy]{Definition}

\theoremstyle{definition}
\newtheorem{remark}[dummy]{Remark}
\newtheorem{ex}[dummy]{Example}

\Crefname{ex}{Example}{Examples}	% to get the plural to appear
\Crefname{thm}{Theorem}{Theorems}
\Crefname{prop}{Proposition}{Propositions}
\Crefname{cor}{Corollary}{Corollaries}
\Crefname{defn}{Definition}{Definitions}

\newcommand{\N}{\mathbb{N}}

\newcommand{\R}{\mathbb{R}}

\newcommand{\cat}[1]{{\mathsf{#1}}} % font for categories
\newcommand{\ar}[2][]{\arrow{#2}{#1}}
\DeclareMathOperator*{\minl}{min}
\newcommand{\id}[1][]{\ifthenelse{\equal{#1}{}}{\mathrm{id}}{\mathrm{id}_{#1}}} % identity, with object as subscript if optional argument is given
\newcommand{\op}{\mathrm{op}}

\newcommand{\atol}{\xrightarrow}

\newcommand{\blp}{\Big{(}}
\newcommand{\brp}{\Big{)}}

\DeclareMathOperator{\im}{im}
\newcommand{\scdots}[2][]{\mathinner{#1\overset{#2}{\cdots}#1}}
\newcommand*{\LargerCdot}{\raisebox{-0.25ex}{\scalebox{1.5}{$\cdot$}}}

\newcommand{\C}{{\cat{C}}}

\newcommand{\Set}{\cat{Set}}
\newcommand{\Cat}{\cat{Cat}}

\newcommand{\Sp}{{\cat{Sp}}}

\newcommand{\Ex}{\mathsf{Ex}}
\newcommand{\Comp}{\mathsf{Comp}}

\let\originalleft\left
\let\originalright\right
\renewcommand{\left}{\mathopen{}\mathclose\bgroup\originalleft}
\renewcommand{\right}{\aftergroup\egroup\originalright}

\tikzset{% 
    bullet/.style={
       fill=black,
       circle,
       minimum width=1pt,
       inner sep=1pt
     },
     relation/.style={
       -,
       thick,
       shorten <=2pt,
       shorten >=2pt
     },
     function/.style={
       ->,
       thick,
       shorten <=2pt,
       shorten >=2pt
     },
     every fit/.style={
       ellipse,
       draw,
       inner sep=0pt
     }
}

\usepackage{enumitem}
\setlist[enumerate]{label=(\alph*),itemsep=5pt,topsep=8pt}
\setlist[itemize]{label=$\triangleright$,itemsep=5pt,topsep=6pt}

\Crefformat{enumi}{#2#1#3}
\Crefname{equation}{}{}		% abbreviate 'Equation (3.2)' to '(3.2)'

\author[C.~Constantin]{Carmen Constantin}
\address{Mansfield College, Oxford, UK}
\email{carmen.constantin@mansfield.ox.ac.uk}

\author[T.~Fritz]{Tobias Fritz}
\address{Department of Mathematics, University of Innsbruck, Austria}
\email{tobias.fritz@uibk.ac.at}

\author[P.~Perrone]{Paolo Perrone}
\address{Department of Computer Science, University of Oxford, United Kingdom}
\email{paolo.perrone@cs.ox.ac.uk}

\author[B.~Shapiro]{Brandon T. Shapiro}
\address{Topos Institute, Berkeley CA, U.S.A}
\email{shapiro@topos.institute}

\usepackage{tikz}

\title[Weak cartesian properties of simplicial sets]{Weak cartesian properties of simplicial sets \vspace{3pt}}

\begin{document}

\begin{abstract}
	\vspace{3pt}

	Many special classes of simplicial sets, such as the nerves of categories or groupoids, the 2-Segal sets of Dyckerhoff and Kapranov, and the (discrete) decomposition spaces of G\'{a}lvez, Kock, and Tonks, are characterized by the property of sending certain commuting squares in the simplex category $\Delta$ to pullback squares of sets. We introduce weaker analogues of these properties called completeness conditions, which require squares in $\Delta$ to be sent to weak pullbacks of sets, defined similarly to pullback squares but without the uniqueness property of induced maps. We show that some of these completeness conditions provide a simplicial set with lifts against certain subsets of simplices first introduced in the theory of database design. We also provide reduced criteria for checking these properties using factorization results for pushouts squares in $\Delta$, which we characterize completely, along with several other classes of squares in $\Delta$. Examples of simplicial sets with completeness conditions include quasicategories, many of the compositories and gleaves of Flori and Fritz, and bar constructions for algebras of certain classes of monads. The latter is our motivating example.
\end{abstract}

%%% title page, newgeometry makes the title be placed higher up so that the abstract and table of contents fit onto the page
\newgeometry{top=0.2cm,bottom=1cm}
\maketitle
\thispagestyle{empty}
\setcounter{tocdepth}{1}
\tableofcontents
\vspace{-1cm}

\restoregeometry

\newpage

\section{Introduction}

Compositional structures, such as categories, are most commonly defined in terms of algebraic operations satisfying certain algebraic laws. But there is a powerful alternative picture in which compositional structures are considered as certain combinatorial structures with merely extra properties (and no algebraic structure). For example, taking the \emph{nerve} of a category produces a combinatorial structure---in the form of a simplicial set---from which the category can be recovered; and conversely, every simplicial set which has a property known as the \emph{Segal condition} encodes a category under this correspondence.

Moreover, this combinatorial perspective suggests far-reaching generalizations of the concept of category, obtained by suitable relaxations of the Segal condition. The most prominent of these is the generalization to \emph{quasicategories}~\cite{joyal}, which are defined as simplicial sets with the additional property that every configuration consisting of $n$ $n$-simplices with shared faces arranged in the shape of an \emph{inner horn}, can be obtained from the faces of an $(n+1)$-simplex, called a \emph{filler} of the inner horn. Since the filler is merely required to exist and is generally non-unique, quasicategories are a compositional structure that is not algebraic.\footnote{It is worth noting that the theory of $(\infty,1)$-categories, which is modeled by quasicategories, also has closely related algebraic models~\cite{nikolaus}.}
Another example that has gained prominence recently is that of \emph{$2$-Segal sets}~\cite{DK2segal}, also known as (the discrete case of) \emph{decomposition spaces}~\cite{GKT1}. These again are compositional structures coming in the form of simplicial sets satisfying certain (in this case unique) filler conditions~\cite{DK2segal}. They arise naturally in many different ways in combinatorics~\cite{DK2segal,GKT1}.

In this paper, we introduce new compositional structures defined in such combinatorial terms. These structures are motivated by \emph{partial evaluations}~\cite{FP}, which is the idea that an algebraic expression like $3 + 1 + 4$ can not only be ``totally'' evaluated to $8$, but it can also be ``partially'' evaluated to $3 + 5$. This is formalized in terms of the bar construction of Eilenberg-Moore algebras of monads; on a concrete category, this bar construction results in a simplicial set whose $1$-skeleton describes partial evaluations. As shown in~\cite{FP}, these partial evaluations can often be composed (non-uniquely). This naturally raises the question whether the entire bar construction, considered as a simplicial set, has properties which encode a compositional structure similar to quasicategories or $2$-Segal sets. As we show in our companion paper~\cite{weakly_cartesian}, the bar construction of many monads, including many of those which describe commonly occurring algebraic structures, display the compositional properties that we study in this paper.

We start in \Cref{sec:vee_stuff,sec:pushout_char} with a thorough study of pushouts in the simplex category $\Delta$. This is based on the \emph{$\vee$-decomposition} of objects, morphisms, and more general diagrams in $\Delta$ that we introduce in \Cref{sec:vee_stuff}. The associated \emph{$\vee$-product} on $\Delta$ amounts to a partially defined monoidal structure which glues two morphisms $f_1 : [n_1] \to [m_1]$ and $f_2 : [n_2] \to [m_2]$ to
\[
	f_1 \vee f_2 \: : \: [n_1 + n_2] \longrightarrow [m_1 + m_2]
\]
whenever $f_1$ preserves the final vertex, as in $f_1(n_1) = m_1$, and $f_2$ preserves the initial vertex, as in $f_2(0) = 0$. By making judicious use of the $\vee$-decomposition and $\vee$-product, we obtain several characterization results on pushouts in $\Delta$. \Cref{thm:pushouts_delta_star} characterizes pushout squares as $\vee$-products of four minimal types of pushout squares with $[0]$ or $[1]$ in the upper left corner. \Cref{thm:pushouts_delta} then characterizes when a span in $\Delta$ has a pushout at all, and \Cref{prop:all_basic_pushouts} factors every pushout square in $\Delta$ into certain \emph{basic pushouts} and trivial pushouts. \Cref{lem:basic_decompose} lists $8$ particular pushout diagrams in $\Delta$ such that all others arise from these by composition and $\vee$-products, with all but two of these squares having parallel identity morphisms in at least one direction. \Cref{prop:balanced_factorization,balanced_decompose} provide similar results for \emph{balanced squares} in $\Delta$, which are those squares of coface maps that are pushouts of finite sets (but not necessarily pushouts in $\Delta$).

In \Cref{sec:compl_exact}, we consider the classes of squares in $\Delta$ that a given simplicial set $X : \Delta^\op \to \Set$ sends to (weak) pullbacks. Our philosophy is that this encodes compositional properties enjoyed by $X$: we consider a square in $\Delta$ being sent to a weak pullback a \emph{completeness} property, amounting to the existence of certain fillers, while being sent to a pullback is an \emph{exactness} property where the fillers are in addition unique. Based on the previous results on the characterization of pushouts in $\Delta$, \Cref{complete_characterization,exact_characterization} then reduce completeness and exactness with respect to entire classes of squares to simpler ones. 

Thus for every class of squares in $\Delta$, postulating completeness or exactness with respect to these squares specifies a type of compositional structure defined in terms of filler conditions. For example, $2$-Segal sets can be defined in this way~\cite[Proposition 2.3.2]{DK2segal}. Of particular interest to us in the context of the bar construction~\cite{weakly_cartesian} are simplicial sets that we call \emph{inner span complete}. A \emph{span complete} simplicial set is one which sends all balanced squares of coface maps to weak pullbacks (\Cref{def:span_complete}), while an inner span complete simplicial set only needs to send pushout squares of coface maps to weak pullbacks (\Cref{def:inner_span_complete}). These are characterized by the possibility to fill any pair of $(n-1)$-simplices that overlap on an $(n-2)$-simplex face to an $n$-simplex (for all $n \ge 2$), with an ``innerness'' restriction of these pairs in the inner span case analogous to the restriction of horns to inner horns when generalizing Kan complexes to quasicategories.

While these compositional properties may sound rather weak, we show in \Cref{acyclic_fillers,acyclicfiller} that they are sufficient to imply the existence of much more general fillers, namely fillers for all \emph{(directed) acyclic configurations}. 
Among the most basic instances of this is the consequence that any string of $1$-simplices of length $n$, as in the Segal condition, has an $n$-simplex filler.
In general, these acyclic configurations are (directed) simplicial complexes characterized in \Cref{undirected_acyclic,directed_acyclic} in terms of combinatorial acyclicity conditions that are not homotopy invariant, but are closely related to notions of shellability and collapsibility in combinatorial topology. While our definition of directed combinatorial acyclicity is new, the undirected version has a long history in database theory~\cite{maier}.

Finally, \Cref{sec:examples} concludes the paper with a presentation of some first examples of (inner) span complete simplicial sets unrelated to the bar construction examples considered in~\cite{weakly_cartesian}. We note in \Cref{qcat_innerspans} that quasicategories are inner span complete, while the converse is not true. We then relate (inner) span completeness to the compositories and gleaves from~\cite{FF}, noting that some of the examples considered there are also span complete or inner span complete simplicial sets. This includes an inner span complete simplicial set of higher spans in any category, a span complete simplicial set where the $n$-simplices are joint probability distributions of $n+1$ random variables, and a closely related one in which they are the tables with $n+1$ columns in a relational database.

\subsection*{Weakly cartesian squares}\label{weakly_cartesian}%add in composition of strong/weak pullbacks?

We now present some basic background in weak pullbacks.

One of the main ideas we consider is replacing definitions involving pullback squares with analogues using instead \emph{weak pullback squares}, which have a weaker universal property than pullbacks (which we sometimes call \emph{strong} pullbacks for emphasis) in that induced maps need not be unique.

\begin{defn}\label{weak pullback}
 (\cite{JoyalAnalytic}) A diagram 
 \begin{equation}\label{square}
 \begin{tikzcd}
	A \ar{r}{f} \ar{d}[swap]{g}  & B \ar{d}[swap]{m} \\
	C \ar{r}{n} & D
 \end{tikzcd}
 \end{equation}
 in a category $\cat{C}$ is called a \emph{weak pullback}, or \emph{weakly cartesian square}, if for every object $S$ and every commutative diagram
 $$
 \begin{tikzcd}[sep=small]
	S \ar{drrr}{p} \ar{dddr}[swap]{q} \\
	& && B \ar{dd}[swap]{m} \\ \\
	& C \ar{rr}{n} && D
 \end{tikzcd}
 $$
 in $\cat{C}$ there exists an arrow $S\to A$ making the following diagram commute. 
 $$
 \begin{tikzcd}[sep=small]
	S \ar{drrr}{p} \ar{dr} \ar{dddr}[swap]{q} \\
	& A \ar{rr}[swap]{f} \ar{dd}{g}  && B \ar{dd}[swap]{m} \\ \\
	& C \ar{rr}{n} && D
 \end{tikzcd}
 $$
\end{defn}

If we are in the category $\cat{Set}$, the diagram~\Cref{square} is a weak pullback if and only if for every $b\in B$ and $c\in C$ with $m(b)=n(c)$ there exists $a\in A$ such that $f(a)=b$ and $g(a)=c$.  Note that if we moreover require the map $S\to A$ to be unique, we recover the ordinary notion of pullback (or cartesian square). 

Like strong pullbacks (and by the same argument), weak pullbacks are closed under horizontal and vertical composition of squares. Strong pullbacks further satisfy the following standard \emph{pullback lemma}, also known as the \emph{prism lemma} in the homotopical setting (see for instance \cite[Lemma 1.11]{GKT1}).

\begin{lemma}\label{pullback_lemma}
In any diagram as below, if the right square and outer rectangle are strong pullbacks, then so is the left square.
$$
\begin{tikzcd}
	\LargerCdot \ar{r} \ar{d} & \LargerCdot \ar{r} \ar{d} & \LargerCdot \ar{d} \\
	\LargerCdot \ar{r} & \LargerCdot \ar{r} & \LargerCdot
\end{tikzcd}
$$
\end{lemma}

A fundamental difference between strong and weak pullbacks is that this does not hold for weak pullbacks in general.

\begin{ex}
Consider the diagram below in $\Set$:
$$
\begin{tikzcd}
	\{*\} \ar{drr}{b} \ar[equals]{ddr} \\
	& \{*\} \ar{r}[near start]{a} \ar[equals]{d} & \{a,b\} \ar{r} \ar{d} & \{*\} \ar[equals]{d} \\
	& \{*\} \ar[equals]{r} & \{*\} \ar[equals]{r} & \{*\}
\end{tikzcd}
$$

Both the right square and the outer rectangle are weak pullbacks, and the kite shaped subdiagram commutes, but there is no map $h : \{*\} \to \{*\}$ with $ah=b$. The left square is therefore not a weak pullback.
\end{ex}

In categories with all pullbacks such as $\Set$, the following is a useful characterization of weak pullback squares, which follows immediately from considering the induced maps in both directions between weak and strong pullbacks of the same cospan.

\begin{lemma}\label{pullback_split_epi}
A commutative square in a category with all pullbacks is a weak pullback if and only if the induced map into the pullback of its cospan is split epic.
\end{lemma}

The following lemma will be particularly useful when $f$ or $g$ is a degeneracy map of a simplicial set, which is always (split) monic.

\begin{lemma}\label{mono_strong_pullback}
If the square below is a weak pullback and $f$ or $g$ is monic, then the square is a strong pullback.
\[
 \begin{tikzcd}
	A \ar{r}{f} \ar{d}[swap]{g}  & B \ar{d}[swap]{} \\
	C \ar{r}{} & D
 \end{tikzcd}
\]
\end{lemma}

\begin{proof}
Assume $f$ is monic (the argument for $g$ is analogous), and let $p : S \to B$, $q : S \to C$ be maps which commute over $D$. Any two induced maps $h,h' : S \to A$ with $fh=fh'=p$ are equal as $f$ is monic.
\end{proof}

\subsection*{Simplicial terminology}

Throughout the paper, $\Delta$ denotes the simplex category, i.e.~the category of nonempty finite ordinals
\[
	{[n]} \coloneqq \{0,\ldots,n\}	
\]
for $n \in \N$ as objects and monotone maps as morphisms. Its \emph{generating coface maps} are the morphisms
\[
	d^{n,i} \: : \: {[n-1]} \longrightarrow {[n]}
\]
for $i = 0,\ldots,n$, given by the inclusion of ${[n-1]}$ into ${[n]}$ omitting the element $i$. The \emph{generating codegeneracy maps} are likewise the morphisms
\[
	s^{n,i} \: : \: {[n+1]} \longrightarrow {[n]}
\]
for $i = 0,\ldots,n$, given by the map which hits $i$ twice but otherwise acts like the identity. A coface map or codegeneracy map in general is a composite of generating ones. 

A simplicial set is then a functor $\Delta^\op \to \Set$. As usual, when the simplicial set under consideration is clear from the context, then we denote the face and degeneracy maps (the functor's action on cofaces and codegeneracies) using subscripts, $d_{n,i}$ and $s_{n,i}$, or merely $d_i$ and $s_i$.

When discussing commuting squares in $\Delta$, we implicitly identify a square as below left with its \emph{mirror image} as below right:
\[
	\begin{tikzcd}
		{[m]} \arrow{r}{f} \arrow{d}[swap]{g}	& {[p]} \ar{d}[swap]{h} \\
		{[q]} \arrow{r}{k}			& {[n]}
	\end{tikzcd}\qquad\qquad	\begin{tikzcd}
		{[m]} \arrow{r}{g} \arrow{d}[swap]{f}	& {[q]} \ar{d}[swap]{k} \\
		{[p]} \arrow{r}{h}			& {[n]}
	\end{tikzcd}
\]

We call a square \emph{trivial} if either pair of parallel arrows are identities, and note that trivial squares are automatically pushouts.

\subsection*{Acknowledgements}

We first of all thank Joachim Kock for detailed comments on an earlier version, which have resulted in various improvements to the exposition.

This paper originates from the \emph{Applied Category Theory 2019} school. We thank the organizers Daniel Cicala and Jules Hedges for having made it happen, the Computer Science Department of the University of Oxford for hosting the event, as well as all other participants of the school for the interesting discussions and insights, especially Martin Lundfall. 

Research at Perimeter Institute is supported in part by the Government of Canada through the Department of Innovation, Science and Economic Development Canada and by the Province of Ontario through the Ministry of Colleges and Universities. Research for the third author was partly funded by the Fields Institute (Canada), and by the AFOSR grants FA9550-19-1-0113 and FA9550-17-1-0058 (U.S.A.). The fourth author was supported by the National Defense Science and Engineering Graduate Fellowship Program.

%\section{Pushouts in $\Delta$ and their decomposition}
%\label{sec:pushouts_delta}
%
%
%Recall that we denote the objects of $\Delta$ as ${[n]} = \{0,\ldots,n\}$. Throughout this section, we also use the term ``joint image'' to refer to the union of the images of several maps, and we often say ``strong pushout'' instead of ``pushout'' in order to emphasize the difference with ``weak pushout'', as we have already done for pullbacks.
%
%To illustrate the subtlety of pushouts in $\Delta$, consider the following commutative square in $\Delta$.
%\begin{center}\begin{tikzcd}
%	{[1]} \arrow{r}{02} \arrow{d}[swap]{} & {[2]} \arrow{d}[swap]{} \\
%	{[0]} \arrow[equals]{r}{} & {[0]} 
%\end{tikzcd}\end{center}
%This square is a pushout in $\Delta$, although its image in $\Set$, under the usual forgetful functor $\Delta([0],-) : \Delta \to \Set$, is not. It is useful to have a name for all those pushouts which do not display this (initially) counterintuitive behavior.
%

\section{$\vee$-Decompositions and $\vee$-products in the simplex category}
\label{sec:vee_stuff}

With the goal of reducing the description of pushouts and other classes of squares in $\Delta$ to simpler cases in mind,
we now introduce our main technical tool of decomposing diagrams in $\Delta$ into families of simpler diagrams of the same shape: the $\vee$-decomposition. The key observation is that a morphism $f : {[r]} \to {[n]}$ in $\Delta$ \emph{decomposes} ${[n]}$ into $r+2$ pieces, namely the subsimplices 
\begin{align*}
	{[n_0]} & := \{0,\ldots ,f(0)\}, \\
	{[n_i]} & := \{f(i-1),\ldots ,f(i)\} \quad \text{\normalfont{for }} 1 \le i \le r, \\
	{[n_{r+1}]} & := \{f(r),\ldots ,n\},
\end{align*}
and conversely that these simplices assemble into ${[n]}$ by an operation we call the $\vee$-product, which behaves like a partially defined monoidal structure. We will subsequently exploit the fact that both the $\vee$-decomposition and the $\vee$-product can be expressed as colimits in order to argue that they are well-behaved with respect to pushouts.

\begin{defn}
Let ${[r]} / \Delta$ denote the \emph{undercategory} of ${[r]}$ in $\Delta$, whose objects are maps ${[r]} \to {[n]}$ in $\Delta$ and whose morphisms $({[r]} \to {[n]}) \to ({[r]} \to {[m]})$ are maps ${[n]} \to {[m]}$ commuting with the maps from ${[r]}$. We denote by $\vee : {[r]}/\Delta \to \Delta$ the forgetful functor sending ${[r]} \to {[n]}$ to ${[n]}$ and forgetting the commuting property of the morphisms.
\end{defn}

As alluded to above, we will think of a map $f : {[r]} \to {[n]}$ as a decomposition of ${[n]}$ into $r + 2$ pieces, which we call \emph{$\vee$-components}. We recall now a general categorical property of undercategories, which we make extensive use of in this section.

\begin{lemma}
	\label{lem:opfibration}
	$\vee$ is a discrete opfibration.  That is, for any map $f : {[n]} \to {[m]}$ in $\Delta$ and a lift of ${[n]}$ to $g : {[r]} \to {[n]}$ in ${[r]} / \Delta$, there is a unique lift of ${[m]}$ to $h : {[r]} \to {[m]}$ in ${[r]} / \Delta$ such that $f$ lifts to a map from $g$ to $h$ in ${[r]} / \Delta$.
\end{lemma}

\begin{proof}
Define $h$ to be the composite $fg$ and this follows immediately.  Concretely, this decomposes ${[m]}$ into the $\vee$-components
\begin{align*}
	{[m_0]} & := \{0,\ldots ,f(g(0))\}, \\
	{[m_i]} & := \{f(g(i-1)),\ldots ,f(g(i))\} \quad \text{\normalfont{for }} 1 \le i \le r, \\
	{[m_{r+1}]} & := \{f(g(r)),\ldots ,m\}. \qedhere
\end{align*}
\end{proof}

We will call a lift of ${[n]}$ to ${[r]} / \Delta$ for some $r$ a \emph{$\vee$-decomposition} of ${[n]}$.
The lemma shows that $\vee$-decompositions push forward along maps in $\Delta$.  This lets us further extend a $\vee$-decomposition on ${[n]}$ to more general diagrams in $\Delta$.

\begin{cor}
	\label{cor:diagram_star_decomp}
	Let $\cat{J}$ be a category with an initial object $I$, and $D : \cat{J} \to \Delta$ a diagram. Then every $\vee$-decomposition of $D(I)$ extends uniquely to a $\vee$-decomposition of the whole diagram $D$; that is, $D$ lifts along $\vee$ to a diagram $D' : \cat{J} \to {[r]} / \Delta$.
\end{cor}

\begin{proof}
	This is a standard property of discrete opfibrations.
\end{proof}

This allows us to $\vee$-decompose spans and squares in $\Delta$ according to a $\vee$-decomposition of their initial object. For our purposes, we will take the $\vee$-decomposition of $D(I)$ identifying separately its endpoints and each of its edges as follows.

\begin{defn}
The \emph{canonical $\vee$-decomposition} of ${[n]}$ is $\id : {[n]} \to {[n]}$ in ${[n]} / \Delta$, or equivalently the expression of ${[n]}$ as ${[0]} \vee {[1]} \vee {[1]} \vee \cdots \vee {[1]} \vee {[0]}$ (with $n$ copies of ${[1]}$).
\end{defn}

For $D$ a diagram as above, the \emph{canonical $\vee$-decomposition of $D$} is the decomposition induced by the canonical $\vee$-decomposition of $D(I)$.

$\vee$ can also be expressed as a colimit.

\newcommand{\Fun}[2]{\mathrm{Fun}_\vee(#1,#2)}	% diagram cat

\begin{lemma}
\label{lem:star_as_colim}
Consider diagrams in $\Delta$ as below, where we specify a morphism out of the singleton set ${[0]} = \{0\}$ by its image in the target:

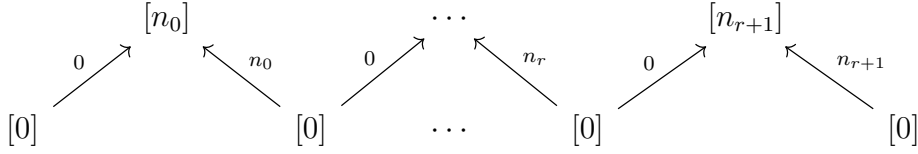
\begin{center}\begin{tikzcd}
	{[n_0]} & & \cdots & & {[n_{r+1}]} \\
	& {[0]} \arrow{ul}[swap]{n_0} \arrow{ur}{0} & \cdots & {[0]} \arrow{ul}[swap]{n_r} \arrow{ur}{0} 
\end{tikzcd}\end{center}
We denote the shape of these diagrams by $\cat{V}^{r+1}$. Then:

\begin{enumerate}
\item ${[r]} / \Delta$ is isomorphic to the category $\Fun{\cat{V}^{r+1}}{\Delta}$ of diagrams of this form and natural transformations.
\item The functor $\Fun{\cat{V}^{r+1}}{\Delta} \cong {[r]} / \Delta \atol{\vee} \Delta$ is naturally isomorphic to the colimit functor sending such a diagram to its colimit ${[n_0 + \cdots + n_{r+1}]}$.
\end{enumerate}
\end{lemma}

\begin{proof}
\begin{enumerate}
	\item Given such a diagram, construct a map $g : {[r]} \to {[n_0 + \cdots + n_{r+1}]}$ sending $i$ to $n_0 + \cdots + n_i$ for $0 \le i \le r$.  Conversely given a map $g : {[r]} \to {[n]}$, construct such a diagram by setting
\begin{align*}
	n_0 & := g(0) \\
	n_i & := g(i) - g(i-1) \quad \text{\normalfont{for }} 1 \le i \le r, \\
	n_{r+1} & := n - g(r).
\end{align*}
These constructions are easily checked to be inverse to one another, defining a bijection between objects in ${[r]} / \Delta$ and $\Fun{\cat{V}^{r+1}}{\Delta}$.  

Natural transformations in $\Fun{\cat{V}^{r+1}}{\Delta}$ from the diagram given by $({[n_0]},\ldots,{[n_{r+1}]})$ to another one given by $({[m_0]},\ldots,{[m_{r+1}]})$ correspond to tuples of maps $f_i : {[n_i]} \to {[m_i]}$ such that $f_0,\ldots,f_r$ preserve the maximum element and $f_1,\ldots,f_{r+1}$ preserve the minimum element. Morphisms in ${[r]} / \Delta$ from $g : {[r]} \to {[n]}$ to $h : {[r]} \to {[m]}$ amount to a family of monotone maps like this:
\begin{align*}
	f_0 \: : \: & \{0,\ldots,g(0)\} & \longrightarrow {} & \{0,\ldots,h(0)\} \\
	f_i \: : \: & \{g(i-1),\ldots,g(i)\} & \longrightarrow {} & \{h(i-1),\ldots,h(i)\} \quad \text{\normalfont{for }} 1 \le i \le r, \\
	f_{r+1} \: : \: & \{g(r),\ldots,n\} & \longrightarrow {} & \{h(r),\ldots,m\},
\end{align*}
such that $f_0,\ldots,f_r$ preserve the maximum element and $f_1,\ldots,f_{r+1}$ preserve the minimum element. These two types of morphisms are in an obvious bijection, matching the bijection on objects defined above and preserving composition.

\item The composite functor $\Fun{\cat{V}^{r+1}}{\Delta} \cong {[r]} / \Delta \atol{\vee} \Delta$ indeed sends the pictured diagram to ${[n_0 + \cdots + n_{r+1}]}$, so it remains to show that this is a colimit.

	A cocone from this diagram to some ${[m]}$ consists of $r+1$ elements $x_0 \le \ldots \le x_r$ in ${[m]}$ and monotone maps $f_i : {[n_i]} \to {[m]}$ for $i=0,\ldots,r+1$ satisfying
\[
	f_{i+1}(0) = x_i, \qquad f_i(n_i) = x_i
\]
whenever $i=0,\ldots,r$. This data uniquely determines a map $f : {[n_0 + \cdots + n_{r+1}]} \to {[m]}$ by defining, for any $i = 0,\ldots,r+1$ and $0 \le j \le n_i$,
\[
	f(n_0+\cdots+n_{i-1}+j) \coloneqq f_i(j),
\]
where the above compatibility conditions between the $f_i$ guarantee that this is well-defined. $f$ is by definition monotone on every subset from $n_0+\cdots + n_{i-1}$ to $n_0+\cdots+n_i$, which implies monotonicity overall. $f$ restricts to $f_i$ along the inclusions ${[n_i]} \to {[n_0+ \cdots + n_{r+1}]}$ sending $0$ to $n_0+\cdots+n_{i-1}$ and $n_i$ to $n_0 + \cdots + n_i$. Since these inclusions are moreover jointly surjective, this property uniquely determines $f$. \qedhere
\end{enumerate}
\end{proof}

This equivalent perspective motivates the following alternative notation for $\vee$.

\begin{defn}
	For any finite sequence $n_0,\ldots,n_{r+1} \in \N$, define the \emph{$\vee$-product} ${[n_0]} \vee \cdots \vee {[n_{r+1}]}$ as ${[n_0+ \cdots + n_{r+1}]}$.  Furthermore, for maps $f_i : {[n_i]} \to {[m_i]}$ in $\Delta$ with $i=0,\ldots,r+1$, their $\vee$-product
\begin{equation}
	\label{eq:star_product_morphisms}
	f_0 \vee \cdots \vee f_{r+1} \: : \: {[n_0]} \vee \cdots \vee {[n_{r+1}]} \longrightarrow {[m_0]} \vee \cdots \vee {[m_{r+1}]}
\end{equation}
is defined (as above) precisely when the $f_0,\ldots,f_r$ preserve maximum elements and the $f_1,\ldots,f_{r+1}$ preserve minimum elements.
\end{defn}

More generally, the $\vee$-product of a finite sequence of diagrams of the same shape exists precisely when all morphisms in the diagrams satisfy these preservation conditions. In the binary case, $\vee$ defines a functor $\Delta_{\max} \times \Delta_{\min} \to \Delta$, where $\Delta_{\max},\Delta_{\min}$ are the subcategories of $\Delta$ containing all maps which preserve the maximal (resp.~minimal) element. The intersection of these subcategories, containing maps which preserve both endpoints, is the category of active maps $\Delta_{\mathrm{act}}$ of \cite[2.4]{GKT1}. The restriction of our $\vee$ to $\Delta_{\mathrm{act}} \times \Delta_{\mathrm{act}}$ is precisely the amalgamated ordinal sum functor $\vee$ in that setting, which by \cite[Lemma 6.2]{GKT1} agrees with the ordinal sum on $\Delta_+^{\op} \cong \Delta_{\mathrm{act}}$.  In this sense, it is appropriate to view $\vee$ as an ordered sum of the \emph{edges}, not vertices, of the ordinals $[n]$ in $\Delta$. It is straightforward to check that $\vee$ is unital (with respect to $[0]$) and associative in the appropriate senses.

We can also express the extraction of each $\vee$-component as a colimit.  While perhaps the more intuitive relationship between a $\vee$-product and its components is the inclusion ${[n_i]} \to {[n_0]} \vee \cdots \vee {[n_{r+1}]}$, more helpful for proving that $\vee$-products reflect pushouts is the surjective map ${[n_0]} \vee \cdots \vee {[n_{r+1}]} \to {[n_i]}$ acting as the identity on the component ${[n_i]}$ and as the constant map to $0$ or $n_i$ on the components ${[n_j]}$ for $j < i$ or $i < j$ respectively.

\begin{lemma}
\label{lem:component_as_colim}
Given $g : {[r]} \to {[n]}$ exhibiting ${[n]}$ as ${[n_0]} \vee \cdots \vee {[n_{r+1}]}$ and $0 \le i \le r + 1$, let
\[
	g_i^-,g_i^+ \: : \: {[1]} \longrightarrow {[n]}
\]
be the maps with
\begin{align*}
	g_i^-(0) & = 0,				& g_i^+(0) & = n_0 + \cdots + n_i,	\\
	g_i^-(1) & = n_0+\cdots+n_{i-1},	& g_i^+(1) & = n.
\end{align*}
Then the $\vee$-component ${[n_i]}$ is the colimit of the following diagram.
 $$
\begin{tikzcd}[column sep=0]
	& {[1]} \arrow{dl} \arrow{dr}{g_i^-} & & {[1]}  \arrow{dl}[swap]{g_i^+} \arrow{dr} \\
	{[0]} & & {[n_0]} \vee \cdots \vee {[n_{r+1}]} & & {[0]}
\end{tikzcd}
$$
\end{lemma}

\begin{proof}
A cocone out of this diagram is precisely a map ${[n_0]} \vee \cdots \vee {[n_{r+1}]} \to {[m]}$ constant on each of the subobjects ${[n_0]} \vee \cdots \vee {[n_{i-1}]}$ and ${[n_{i+1}]} \vee \cdots \vee {[n_{r+1}]}$. These maps are in obvious bijection with maps ${[n_i]} \to {[m]}$, so ${[n_i]}$ is the colimit of the diagram.
\end{proof}

\section{Pushout squares in $\Delta$}
\label{sec:pushout_char}

We provide three different characterizations of the pushout squares in $\Delta$, first in terms of $\vee$-products, then in terms of composition, and lastly a combination of the two. We then use our techniques to additionally characterize those squares of coface maps which are sent to pushouts by the forgetful functor $\Delta \to \Set$.

\subsection*{Pushouts via $\vee$-products}

With the machinery of $\vee$-decompositions and $\vee$-products in place, we can now apply it to pushouts.

\begin{prop}
\label{prop:preservation_reflection}
Let 
	\[
		\begin{tikzcd}
			{[m_i]} \ar{r}{f_i} \ar{d}[swap]{g_i}	& {[p_i]} 	\\
			{[q_i]} 			& 
		\end{tikzcd}
	\]
	for $i = 0,\ldots,r+1$ be a sequence of spans in $\Delta$ whose $\vee$-product exists. Then:
	\begin{enumerate}
		\item\label{pushout_product} If the above spans all have pushouts, then the $\vee$-product of these pushout squares exists and is a pushout square for the $\vee$-product span.
		\item\label{pushout_decompose} Conversely, if the $\vee$-product span has a pushout, then so do the above spans, and the $\vee$-product of their pushouts is again the pushout of the $\vee$-product span.
	\end{enumerate}
\end{prop}

\begin{proof}
	We first show that a commuting square
	\[
		\begin{tikzcd}
			{[m]} \ar{r}{f} \ar{d}[swap]{g}	& {[p]} \ar{d}[swap]{h}	\\
			{[q]} \ar{r}{k}		& {[n]}
		\end{tikzcd}
	\]
	with jointly surjective $h$ and $k$ is such that if $f$ and $g$ preserve the maximum element, then so do $h$ and $k$. Indeed since the square commutes, the assumption on $f$ and $g$ implies that it is enough that one of $h$ or $k$ preserves the maximum element and the other one follows. But clearly at least one does since $h$ and $k$ must be jointly surjective. A similar argument shows that if $f$ and $g$ preserve the minimum element, then so do $h$ and $k$. 

	\begin{enumerate}
		\item The $\vee$-product of the pushout squares exists since the relevant preservation conditions are implied by the statement from the previous paragraph. Using the description of $\vee$-products as colimits then shows that the resulting $\vee$-product square is a pushout as well since colimits commute with colimits.

		\item For $0 \leq i \leq r+1$, consider the following diagram $D : \cat{J} \to \cat{Span}(\Delta)$, where $\cat{Span}$ denotes the category of spans and $\cat{J}$ is the shape of the diagram in Lemma~\ref{lem:component_as_colim}. The objects in $\cat{J}$ sent to ${[0]}$ and ${[1]}$ in Lemma~\ref{lem:component_as_colim} are sent by $D$ to the constant spans at ${[0]}$ and ${[1]}$ respectively, and the object sent to ${[n_0]} \vee \cdots \vee {[n_{r+1}]}$ in Lemma~\ref{lem:component_as_colim} is sent to the $\vee$-product span of the postulated sequence, with the analogous maps as in Lemma~\ref{lem:component_as_colim} for each of $q,m,p$.  
		
			Each of these spans has a pushout, with the constant spans pushing out to ${[0]}$ and ${[1]}$ respectively and the $\vee$-product span having a pushout by assumption. The functor $\cat{J} \to \Delta$ picking out the pushout objects has a colimit since it is of the form in Lemma~\ref{lem:component_as_colim}, selecting the $i$th component of the pushout.  The diagram $D$ therefore has an overall colimit. As colimits commute with colimits, this means that the $i$th component of the pushout of the $\vee$-product span is the pushout of the $i$th span above.	
		\qedhere
	\end{enumerate}
\end{proof}

This lets us reduce the characterization of pushouts in $\Delta$ to pushouts among the minimal $\vee$-components, the squares in which $m$ as in the square above is $0$ or $1$.
This is our first result on the decomposition of pushouts in $\Delta$.

%mention canonical $\vee$ decomposition of a square?

\begin{thm}
	\label{thm:pushouts_delta_star}
	A span in $\Delta$ has a pushout if and only if its canonical $\vee$-decomposition is made up of the spans in the following pushout squares (and their mirror images), in which case its pushout is the corresponding $\vee$-product of the pushout squares below:
$$
	\begin{tikzcd}
		{[0]} \arrow{r}{p} \arrow[equals]{d}[swap]{} & {[p]} \arrow[equals]{d}[swap]{} \\
		{[0]} \arrow{r}{p} & {[p]}
	\end{tikzcd}\qquad\begin{tikzcd}
		{[1]} \arrow{r}{0p} \arrow[equals]{d}[swap]{} & {[p]} \arrow[equals]{d}[swap]{} \\
		{[1]} \arrow{r}{0p} & {[p]}
	\end{tikzcd}\qquad\begin{tikzcd}
		{[1]} \arrow{r}{0p} \arrow{d}[swap]{} & {[p]} \arrow{d}[swap]{} \\
		{[0]} \arrow[equals]{r}{} & {[0]}
	\end{tikzcd}\qquad\begin{tikzcd}
		{[0]} \arrow{r}{0} \arrow[equals]{d}[swap]{} & {[p]} \arrow[equals]{d}[swap]{} \\
		{[0]} \arrow{r}{0} & {[p]}
	\end{tikzcd}
$$
\end{thm}

Note that for $p = 0$ the first and fourth square coincide, and the second square for $p = 0$ is the mirror image of the third for $p = 1$, but otherwise these squares are all distinct.

\begin{proof}
By \Cref{prop:preservation_reflection}, a span has a pushout if and only if the components of its canonical $\vee$-decomposition have pushouts, which $\vee$ then preserves. It therefore remains to show that the leftmost and rightmost $\vee$-components of a pushout square are always of the forms above left and above right, respectively, and that the middle components are always one of the two middle squares above. 

In the rightmost component of a $\vee$-decomposition, all maps preserve minimal elements, so as the right square in the theorem is a trivial pushout it suffices to show that no square as below is a pushout for $p,q > 0$.
\begin{equation}\label{bad_right_square}
	\begin{tikzcd}
		{[0]} \arrow{r}{0} \arrow{d}[swap]{0} & {[p]} \arrow{d}[swap]{h} \\
		{[q]} \arrow{r}{k} & {[n]}
	\end{tikzcd}
\end{equation}
Without loss of generality we can assume $h(1) \le k(1)$. Define $\phi : {[p]} \to {[1]}$ and $\psi : {[q]} \to {[1]}$ by
\[
	\phi(0) = 0 = \psi(0), \qquad \phi(i) = 1, \qquad \psi(i) = 0, \qquad i > 0.
\]
$\phi,\psi$ commute with the span from $[0]$, but this square does not factor through the putative pushout, as this would require $\phi(1) \le \psi(1)$ by $h(1) \le k(1)$, using monotonicity of the induced map. The argument for the leftmost component is entirely analogous.

For the middle components, we first show that the center right square above is a pushout.
If $\phi : [p] \to [n]$ and $\psi : [0] \to [n]$ as above commute with the span, then $\phi(0) = \phi(p) = \psi(0) \in {[n]}$. But as $\phi$ is monotonic, it must then be constant. Therefore $\phi,\psi$ both factor through $\psi : {[0]} \to {[n]}$, which is unique with respect to this property, hence the square is a pushout.%reword this

It remains then to show that the center left and center right squares above are the only pushout squares with $[1]$ as the source and all maps preserving minimal and maximal elements, as any middle $\vee$-component must. We therefore show that the following square is not a pushout for $p,q > 1$.
\begin{equation}\label{bad_middle_square}
	\begin{tikzcd}
		{[1]} \arrow{r}{0p} \arrow{d}[swap]{0q} & {[p]} \arrow{d}[swap]{h} \\
		{[q]} \arrow{r}{k} & {[n]}
	\end{tikzcd}
\end{equation}
Again assuming $h(1) \le k(1)$, define $\phi : {[p]} \to {[1]}$ and $\psi : {[q]} \to {[1]}$ by
\[
	\phi(0) = 0 = \psi(0), \qquad \phi(p) = 1 = \psi(q), 
\]
\[
	\phi(i) = 1 \quad (0 < i < p),\qquad \psi(j) = 0 \quad (0 < j < q).
\]
$\phi,\psi$ commute with the span from $[1]$, but again there can be no induced map from $[n]$ as by monotonicity this would require $\phi(1) \le \psi(1)$.
\end{proof}

In particular, this construction shows that pushouts of coface maps are again coface maps. %is this especially important?
We also give an elementwise description of this characterization, which follows immediately from the theorem.

\begin{cor}
	\label{thm:pushouts_delta}
	A span
	\[
		\begin{tikzcd}
			{[m]} \arrow{r}{f} \arrow{d}[swap]{g} & {[p]} \\
			{[q]} 
		\end{tikzcd}
	\]
	has a pushout in $\Delta$ if and only if the following three conditions hold:
	\begin{enumerate}
		\item\label{inner_cond} for every $i$ with $1 \le i \le m$, we have $f(i) \le f(i-1) + 1$ or $g(i) \le g(i-1) + 1$;
		\item\label{min_cond} $f(0) = 0$ or $g(0) = 0$;
		\item\label{max_cond} $f(m) = p$ or $g(m) = q$.
	\end{enumerate}
\end{cor}

Property~\ref{inner_cond} fails if $f(i+1) > f(i) + 1$ and $g(i+1) > g(i) + 1$, as $f$ and $g$ should not both ``add an extra element'' in between two consecutive elements of ${[r]}$ as in square~\Cref{bad_middle_square}; the pushout cannot exist as these two elements cannot be totally ordered in a canonical way. The same issue arises when neither $f$ nor $g$ hit the maximum (or minimum) element of their codomains as in square~\Cref{bad_right_square}, in which case~\ref{max_cond} (or ~\ref{min_cond}) fails.
For coface maps, the necessity of having a unique total order on the union of $[p]$ and $[q]$ can be expressed as follows:

\begin{cor}
	\label{face_pushout_concrete}
	A commutative square of coface maps as below is a pushout if and only if it is a pushout in $\Set$, and for every $i = 0,\ldots,n - 1$, the edge $\{i,i+1\} \subseteq [n]$ is in the image of $h$ or $k$.
	\[
		\begin{tikzcd}
			{[m]} \arrow{r}{f} \arrow{d}[swap]{g}	& {[p]} \ar{d}[swap]{h} \\
			{[q]} \arrow{r}{k}			& {[n]}
		\end{tikzcd}
	\]
\end{cor}

We call the second property the \emph{spine condition}. Considering $[n]$ as the geometric $n$-simplex, the extra condition states that $h$ and $k$ must jointly cover the spine of ${[n]}$.

\begin{proof}
The square is a pushout of sets if and only if it induces a bijection
\[
	[n] \;\cong\; \left( [p] \sqcup [q] \right) / \!\sim,
\]
where $\sim$ is the equivalence relation generated by $f(i) \sim g(i)$ for all $i \in [m]$. This is the case for each pushout square of coface maps in \Cref{thm:pushouts_delta_star}, and is preserved by $\vee$-products as the colimits in \Cref{lem:star_as_colim} defining the $\vee$-product are preserved by the forgetful functor $\Delta \to \Set$ and colimits commute with pushouts. The same argument applies to the spine condition, which holds in each pushout square of coface maps in \Cref{thm:pushouts_delta_star} and is preserved by $\vee$-products.

To see that these two conditions are sufficient for the square to be a pushout in $\Delta$, we could show that upon taking the canonical $\vee$-decomposition of the square, these conditions guarantee each component to be of one of the forms in \Cref{thm:pushouts_delta_star}. However, we give a more direct proof demonstrating the uniqueness of the total order on $[n]$ when the square is a pushout of sets satisfying the spine condition.

Let $\phi : [p] \to [n']$, $\psi : [q] \to [n']$ be maps in $\Delta$ satisfying $\phi f = \psi g$. As the square above is a pushout of sets, there is a unique map of sets $\gamma : [n] \to [n']$ with $\phi = \gamma h$ and $\psi = \gamma k$. It remains to show that $\gamma$ is a morphism in $\Delta$, which is to say, that $\gamma$ is monotone. By transitivity it suffices to show that $\gamma(i) \le \gamma(i+1)$ for all $\{i,i+1\} \subseteq [n]$. By the spine condition, noting that $h,k$ are monotone and monic, each such pair $\{i,i+1\}$ lifts along $h$ or $k$ to some $\{j,j+1\}$ in $[p]$ or $[q]$, respectively. We can assume without loss of generality that $\{j,j+1\} \subseteq [p]$, so that as $\phi$ is monotone we have
$$\gamma(i) = \gamma(h(j)) = \phi(j) \le \phi(j+1) = \gamma(h(j+1)) = \gamma(i+1).$$
Therefore, $\gamma$ is monotone and so $[n]$ is a pushout in $\Delta$.
%To see that these two conditions are sufficient for the square to be a pushout in $\Delta$, we show that upon taking the canonical $\vee$-decomposition of the square, these conditions guarantee each component to be of one of the forms in \Cref{thm:pushouts_delta_star}. To analyze each such component, we restrict the square above to the following three cases:
%\begin{itemize}
%	\item If $m = 0$ and all maps are maximum preserving, then if the square is a pushout of sets, then only the maximal element $n \in [n]$ can be in the image of both $h$ and $k$, and each other element belongs to the image of one or the other. Without loss of generality, we can assume $n-1 \in \im(h)$. By the spine condition, $\{n-2,n-1\}$ belongs to either $\im(h)$ or $\im(k)$, but each element can only belong to one or the other, so both must be in $\im(h)$. Proceeding by induction, we see that every element of $[n]$ must then be in $\im(h)$, which shows that $h$ is the identity and $k$ is the inclusion of $[0]$ as the maximal element of $[p]$, just as in the leftmost square of \Cref{thm:pushouts_delta_star}.
%	\item If $m=0$ and all maps are minimum preserving, by a symmetric argument the square is of the form of the rightmost square of \Cref{thm:pushouts_delta_star}.
%	\item If $m=1$ and all maps preserve both minimal and maximal elements, then a similar argument shows that if $n-1 \in [n]$ belongs to $\im(h)$, then the same is true for $1,2,...,n-2$, so the square is of the form of the center left square in \Cref{thm:pushouts_delta_star}. \qedhere
%\end{itemize}
\end{proof}

Taking $[n'] = [n]$ in this proof shows that any jointly surjective $\phi,\psi$ to $[n]$ induces the identity map $[n] \to [n]$, so that $\phi = h$ and $\psi = k$. This is the uniqueness property alluded to above for the order on the union of $[p],[q]$.

\subsection*{Pushouts via composition}

We now proceed to describe another characterization of pushout squares in $\Delta$ using factorization and composition of squares.

\begin{defn}
The {\em defect} $\delta_f$ of a map $f : {[n]} \to {[m]}$ in $\Delta$ is
\[
	\delta_f \coloneqq (|[n]| - |\im(f)|) + (|[m]| - |\im(f)|) = n + m + 2 - 2|\im(f)|.
\]
\end{defn}

The idea of the defect is to measure how far a map in $\Delta$ is from an identity: all identity maps have defect 0, every element of the codomain outside the image adds 1 to the defect, and every pair of adjacent elements in the domain which are identified by $f$ adds 1 to the defect. Conveniently, the defect is additive with respect to $\vee$:

%The defect counts the number of elements in the domain identified by $f$ with another element plus the number of elements in the codomain outside the image. The idea is to measure how far a map in $\Delta$ is from an identity, which has defect 0, in a fairly symmetric way that behaves uniformly across different values of $n$ and $m$.  Conveniently, the defect is additive with respect to $\vee$:

\begin{lemma}
	\label{lem:star_defect_compatible}
	For maps $f : {[n_0]} \to {[m_0]}$ with $f(n_0) = m_0$ and $g : {[n_1]} \to {[m_1]}$ with $g(0) = 0$, we have $\delta_{f \vee g} = \delta_f + \delta_g$.
\end{lemma}

Note that the equations $f(n_0) = m_0$ and $g(0) = 0$ are relevant only for ensuring that the $\vee$-product $f \vee g$ exists.

\begin{proof}
	First observe that $f \vee g : {[n_0+n_1]} \to {[m_0+m_1]}$. Since both the left and right parts of $f \vee g$ have $m_0 \in {[m_0+m_1]}$ in their image, and the images are otherwise disjoint, we have $|\im(f \vee g)| = |\im(f)| + |\im(g)| - 1$.  We then calculate
\begin{align*}
	\delta_{f \vee g}	& = n_0+n_1+m_0+m_1+2-2|\im(f)|-2|\im(g)|+2 \\
				& = (n_0+m_0+2-2|\im(f)|) + (n_1+m_1+2-2|\im(g)|) \\
				& = \delta_f + \delta_g. \qedhere
\end{align*}
\end{proof}

The maps with defect $1$ are exactly the generating coface and codegeneracy maps, since they either identify one pair of elements in the domain or map injectively into a codomain with one additional element.  From this perspective, the defect of $f$ can be seen as counting the minimal number of generating maps in $\Delta$ that $f$ factors into, since each identification in the domain requires a generating codegeneracy and each element in the codomain outside the image requires a generating coface map.  A factorization of $f$ into such a minimal number of generators is what we call {\em efficient}, and in this case the defects of the factors (all $1$) add up to the total defect of $f$. More generally, we declare the following.

\begin{defn}
A factorization $f = h \circ g$ of a map in $\Delta$ is {\em efficient} if $\delta_f = \delta_h + \delta_g$.
\end{defn}

All maps $h$ and $g$ satisfy $\delta_{h \circ g} \le \delta_h + \delta_g$ as $h \circ g$ can, at worst, be factored into the defect 1 maps which generate $h$ and $g$. However, as an example of an inefficient factorization, consider ${[0]} \atol{0} {[2]} \atol{011} {[1]}$, which compose to ${[0]} \atol{0} {[1]}$.  The composite has defect 1, but the factors have respective defects 2 and 1 adding up to 3, hence this factorization is not efficient.  

Any map $f$ has an efficient factorization into $\delta_f$ generating maps as described above; this can be chosen such that the generating coface maps follow the degeneracies, as the Reedy factorization\footnote{See for example \cite[Section~14.2]{catht}.} of a map is efficient.  In fact, let $f = ds$ be the Reedy factorization, so that $s$ is a codegeneracy map and $d$ a coface map. Then $\delta_f$ is the sum of the degree changes of $d$ and $s$, which determine the number of coface and codegeneracy maps in such an efficient factorization of $f$ into generators. %This probably warrants more detail or going into some proposition

For our purposes, efficiency of a factorization guarantees that the factors of a map do not take unnecessarily large steps that could prevent the factorization from extending to pushout squares of the composite.

\begin{prop}
	\label{prop:pushout_factor}
	For a pushout square as below left and an efficient factorization $f = f_1 \circ f_0$, the square factors into a horizontal composite of pushout squares as below right. 
\end{prop}

\begin{center}
\begin{tikzcd}
	{[m]} \arrow{r}{f} \arrow{d}[swap]{g} & {[p]} \arrow{d}[swap]{h}		& & {[m]} \arrow{r}{f_0} \arrow{d}[swap]{g} & {[\ell]} \arrow{r}{f_1} \arrow{d}[swap]{g'} & {[p]} \arrow{d}[swap]{h} \\
	{[q]} \arrow{r}{k} & {[n]}						& & {[q]} \arrow{r}{k_0} & {[\ell']} \arrow{r}{k_1} & {[n]} 
\end{tikzcd}
\end{center}

\begin{proof}
	If the square above left is trivial, as in $g,h$ are identities and $f=k$, then any factorization of $f$ extends to a pair of trivial squares as above right with $g'$ also an identity.

	By \Cref{lem:star_defect_compatible} and \Cref{thm:pushouts_delta_star}, it suffices to check this for the four squares from \Cref{thm:pushouts_delta_star}. Three of those squares are trivial, so we need only consider the following square:
\[
	\begin{tikzcd}
			{[1]} \arrow{r}{0p} \arrow{d}[swap]{} & {[p]} \arrow{d}[swap]{} \\
			{[0]} \arrow[equals]{r}{} & {[0]}
	\end{tikzcd}
\]
An efficient factorization of $[1] \xrightarrow{0p} [p]$ consists of two endpoint preserving coface maps ${[1]} \atol{f_0} {[p']} \atol{f_1} {[p]}$. The vertical map $g' : {[p']} \to {[0]}$ makes the left square a pushout by \Cref{thm:pushouts_delta_star}, and the right square is then a pushout by the pushout lemma (dual to \Cref{pullback_lemma}).
\end{proof}

The defect also plays nicely with pushouts as follows.

\begin{lemma}
	\label{lem:defect_monotone}
	For a pushout square in $\Delta$ as below, assume that $f$ is a coface map or $g$ is a codegeneracy map. Then $\delta_k \leq \delta_f$.%, and $\delta_k = \delta_f$ if and only if the pushout is concrete.
\begin{center}
\begin{tikzcd}
	{[m]} \arrow{r}{f} \arrow{d}[swap]{g} & {[p]} \arrow{d}[swap]{h} \\
	{[q]} \arrow{r}{k} & {[n]} 
\end{tikzcd}
\end{center}
\end{lemma}

\begin{proof}
	By \Cref{lem:star_defect_compatible} and \Cref{thm:pushouts_delta_star}, it again suffices to check this property on the four squares from \Cref{thm:pushouts_delta_star}.  For the three trivial squares parallel maps have the same defect so $\delta_k=\delta_f$, so it suffices to check the two reflections of the remaining square:
\[
	\begin{tikzcd}
			{[1]} \arrow{r}{0p} \arrow{d}[swap]{} & {[p]} \arrow{d}[swap]{} \\
			{[0]} \arrow[equals]{r}{} & {[0]}
	\end{tikzcd}\qquad\qquad\begin{tikzcd}
			{[1]} \arrow{r}{} \arrow{d}[swap]{0p} & {[0]} \arrow[equals]{d}[swap]{} \\
			{[p]} \arrow{r}{} & {[0]}
	\end{tikzcd}
\]
In the left square, whose left map $g$ is a codegeneracy, the top map $f$ has defect $|p-1|$, and the bottom map $k$ has defect $0$, so $\delta_k \leq \delta_f$.  In the right square, the top map $f$ is not a coface map and the left map $g$ is only a codegeneracy if $p=0$, which makes it also of the form of the left square, so we can ignore this case.
%	Concerning the additional statement on $\delta_k = \delta_f$, the above square is a pushout in $\Set$ if and only if $p = 1$. Thus this claim follows also by considering the same cases.
\end{proof}

These results can be combined to prove that any pushout in $\Delta$ can be factored into a grid of pushout squares with spans having both maps generating cofaces or codegeneracies. 
We call these squares {\em basic pushouts}.

\begin{thm}
	\label{prop:all_basic_pushouts}
	A square in $\Delta$ is a pushout if and only if it can be obtained from horizontal and vertical composition of basic pushouts and trivial pushouts of generating maps.
	
%	Furthermore:
%	\begin{enumerate}
%		\item If the original pushout is one of two coface maps, then this can be achieved with only basic pushouts of coface maps.
%		\item If the original pushout is concrete, then this can be achieved with only concrete basic pushouts.
%	\end{enumerate}
\end{thm}

\begin{proof}
The ``if'' direction follows immediately from the fact that pushouts are closed under composition. For the ``only if'' direction, consider a pushout square in $\Delta$ as below
\[
\begin{tikzcd}
	{[m]} \arrow{r}{f} \arrow{d}[swap]{g} & {[p]} \arrow{d}[swap]{h} \\
	{[q]} \arrow{r}{k} & {[n]} 
\end{tikzcd}
\]

First, assume that $g$ is a generating coface map, and factor $f$ efficiently into generating coface and codegeneracy maps $f_1,...,f_{\delta_f}$.  This factorization extends to horizontally factor the pushout square by \Cref{prop:pushout_factor}, as pictured below. By repeated application of \Cref{lem:defect_monotone}, since $\delta_g = 1$ all of the vertical maps $g_i$ and $h$ have defect $1$ or $0$. Therefore, each of the factor squares is either a basic pushout or a trivial pushout of $f_i$. 
\[\begin{tikzcd}
{[m]} \arrow{r}{f_1} \arrow{d}[swap]{g} & {[m_1]} \arrow{r}{f_2} \arrow{d}[swap]{g_1} & \cdots \arrow{r}{f_{\delta_f-1}} & {[m_{\delta_f-1}]} \arrow{r}{f_{\delta_f}} \arrow{d}[swap]{g_{\delta_g-1}} & {[p]} \arrow{d}[swap]{h} \\
{[q]} \arrow{r}{k_1} & {[q_1]} \arrow{r}{k_2} & \cdots \arrow{r}{k_{\delta_f-1}} & {[q_{\delta_f-1}]} \arrow{r}{k_{\delta_f}} & {[n]}
\end{tikzcd}\]

Next, assume that $g$ is a generating codegeneracy map and factor $f$ efficiently into generating codegeneracies $f_1,...,f_\ell$ followed by generating cofaces $f_{\ell+1},...,f_{\delta_f}$.  By \Cref{prop:pushout_factor}, this factorization extends to horizontally factor the pushout square, as pictured above. By the previous case, each square above with $f_i$ a generating coface factors into basic pushouts as desired. By repeated application of \Cref{lem:defect_monotone}, as $g$ has defect 1, for $i \le \ell$ each vertical map $g_i$ has defect either $1$ or $0$, so the leftmost $\ell$ squares are each either a basic or trivial pushout of $f_i$. 

Finally, for an arbitrary pushout square in $\Delta$ as above, factoring $f$ or $g$ efficiently into generators extends to a factorization of the entire square into pushout squares with one map a generating coface or codegeneracy, again by \Cref{prop:pushout_factor}. The previous two cases then show that each of these squares factors into basic pushouts and trivial pushouts of generators, hence so does the entire square. 
\end{proof}

We now list the basic pushout squares, namely the commuting squares in $\Delta$ whose span consists of generating maps and which satisfy the conditions of \Cref{thm:pushouts_delta}. 
\begin{enumerate}[label=(\roman*)]
	\item Pushouts of two generating coface maps are of the form\footnote{Although the diagram still commutes when $i = j - 1$, it is then no longer a pushout, as the single nontrivial $\vee$-component of its span is one of the following: 
		\[
			\begin{tikzcd}[ampersand replacement=\&]
				      {[1]} \& {[0]} \ar[swap]{l}{d^0} \ar{r}{d^0} \& {[1]} 
				\& \& {[2]} \& {[1]} \ar[swap]{l}{d^1} \ar{r}{d^1} \& {[2]}
				\& \& {[1]} \& {[0]} \ar[swap]{l}{d^1} \ar{r}{d^1} \& {[1]}.
			\end{tikzcd}
		\]}
		\begin{equation}
			\label{type2face}
			\begin{tikzcd}
				{[n-2]} \ar{r}{d^i} \ar{d}[swap]{d^{j-1}}		& {[n-1]} \ar{d}[swap]{d^j}		\\
				{[n-1]} \ar{r}{d^i}				& {[n]}				\\
				{} \ar[phantom]{r}{(0 \le i < j - 1 \le n - 1)}		& {}
			\end{tikzcd}
		\end{equation}
	\item Pushouts of one generating coface and one generating codegeneracy map are of the form
		\begin{equation}
			\label{typemixed}
			\begin{tikzcd}
				{[n]} \arrow{r}{d^i} \arrow{d}[swap]{s^{j-1}}	& {[n+1]} \arrow{d}[swap]{s^j}	\\
				{[n-1]} \arrow{r}{d^i}				& {[n]}				\\
				{} \ar[phantom]{r}{(0 \le i < j \le n)}			& {}
			\end{tikzcd}\quad\begin{tikzcd}
				{[n+1]} \arrow{r}{d^{i+1}} \arrow{d}[swap]{s^i}	& {[n+2]} \arrow{d}{s^i s^i}[swap]{s^i s^{i+1}}	\\
				{[n]} \arrow[equals]{r}			& {[n]}			\\
				{} \ar[phantom]{r}{(0 \le i \le n)}			& {}
			\end{tikzcd}\quad\begin{tikzcd}
				{[n]} \arrow{r}{d^{i+1}} \arrow{d}[swap]{s^j}	& {[n+1]} \arrow{d}[swap]{s^j}	\\
				{[n-1]} \arrow{r}{d^i}				& {[n]}				\\
				{} \ar[phantom]{r}{(0 \le j < i \le n)}			& {}
			\end{tikzcd}
		\end{equation}
		%Note that since every generating codegeneracy map $s^j$ is a split epimorphism, the image of such a square under any contravariant functor (such as a simplicial set $X : \Delta^\op \to \Set$) is a pullback already if it is merely a weak pullback, since then the images of the $s^j$ are all monomorphisms.
	\item Pushouts of two generating codegeneracy maps are of the form
		\begin{equation}
			\label{type2deg}
			\begin{tikzcd}
				{[n+2]} \ar{r}{s^i} \ar{d}[swap]{s^{j+1}}	& {[n+1]} \ar{d}[swap]{s^j}		\\
				{[n+1]} \ar{r}{s^i}			& {[n]}				\\
				{} \ar[phantom]{r}{(0 \le i \le j \le n)}	& {}
			\end{tikzcd}\qquad\qquad\begin{tikzcd}
				{[n+1]} \ar{r}{s^i} \ar{d}[swap]{s^i}		& {[n]} \ar[equals]{d}{}		\\
				{[n]} \ar[equals]{r}{}			& {[n]}				\\
				{} \ar[phantom]{r}{(0 \le i \le n)}		& {}
			\end{tikzcd}
		\end{equation}
		%By Theorem 1.2.1 in~\cite{JoyalTierneyNotes}, every pushout of this type is an absolute pushout, meaning it is sent to a pullback by \emph{any} contravariant functor out of $\Delta$.
\end{enumerate}

%It is easily checked by \Cref{lem:defect_monotone} that all of these basic pushout squares are concrete except for the middle squares in~\Cref{typemixed}, which are all $\vee$-products of identities on either side of the motivating non-concrete pushout square at the beginning of this section. % Should probably number this

Immediately from the construction of the factorization in \Cref{prop:all_basic_pushouts}, we can further characterize the following special types of pushouts:

\begin{cor}\label{pure_pushouts}
Consider a pushout square as below in $\Delta$.
\[
\begin{tikzcd}
	{[m]} \arrow{r}{f} \arrow{d}[swap]{g} & {[p]} \arrow{d}[swap]{h} \\
	{[q]} \arrow{r}{k} & {[n]} 
\end{tikzcd}
\]
\begin{enumerate}
	\item If $f,g$ are both coface maps, then the square factors into basic squares of the form in \Cref{type2face}.
	\item If $f,g$ consist of one coface and one codegeneracy, then the square factors into basic squares of the form in \Cref{typemixed}.
	\item If $f,g$ are both codegeneracy maps, then the square factors into basic squares of the form in \Cref{type2deg}.
	\item If $\delta_f=\delta_k$ or $\delta_g=\delta_h$, then the square factors into the squares above that share this property, omitting the middle squares in \Cref{typemixed} and the right squares in \Cref{type2deg}. If $f,g$ further consist of one coface and one codegeneracy, the square factors into the left and right squares in \Cref{typemixed}.
\end{enumerate}
\end{cor}

A pushout square with the property in part $(d)$ is in particular a \emph{concrete pushout}, meaning a pushout preserved by the forgetful functor $\Delta \to \Set$. The right pushout squares in \Cref{type2deg} are also concrete, but the middle squares in \Cref{typemixed} are not.

\begin{remark}\label{2segal_stiff_split_squares}
Related classes of squares in $\Delta$ have been considered in \cite{DK2segal}, \cite{GKT1}, and \cite{GKT2}. In particular, the squares of \cite[Figure (8), Lemma 3.10]{GKT1} sent to pullbacks by a decomposition space include all basic pushouts except for the pushout squares of cofaces in \cite[Lemma 2.10]{GKT1} between outer face maps and the middle squares in \Cref{typemixed}. As decomposition spaces agree with the 2-Segal spaces of \cite{DK2segal} by \cite{2segalunital}, the squares which factor into these restricted basic pushout squares ought to be those pushouts which are preserved by the standard functor from $\Delta$ to Connes' cycle category $\Lambda$, according to \cite[Theorem 2]{Walde}.
\end{remark}

\subsection*{Pushouts via $\vee$-products and composition}\label{subsec:strong_decomposition}

Finally, we can further decompose the basic pushout squares using both $\vee$ and composition.

\begin{cor}
	\label{lem:basic_decompose}
	Pushout squares in $\Delta$ are generated under $\vee$ and composition by the following pushout squares and their mirror images:
	\[
	\begin{tikzcd}
		{[0]} \arrow[equals]{r}{} \arrow[equals]{d}[swap]{} & {[0]} \arrow[equals]{d}[swap]{} \\
		{[0]} \arrow[equals]{r}{} & {[0]} 
	\end{tikzcd}\qquad\begin{tikzcd}
		{[1]} \arrow[equals]{r}{} \arrow[equals]{d}[swap]{} & {[1]} \arrow[equals]{d}[swap]{} \\
		{[1]} \arrow[equals]{r}{} & {[1]} 
	\end{tikzcd}\qquad\begin{tikzcd}
		{[0]} \arrow{r}{d^0} \arrow[equals]{d}[swap]{} & {[1]} \arrow[equals]{d}[swap]{} \\
		{[0]} \arrow{r}{d^0} & {[1]} 
	\end{tikzcd}\qquad\begin{tikzcd}
		{[0]} \arrow{r}{d^1} \arrow[equals]{d}[swap]{} & {[1]} \arrow[equals]{d}[swap]{} \\
		{[0]} \arrow{r}{d^1} & {[1]} 
	\end{tikzcd}
	\]\[
	\begin{tikzcd}
		{[1]} \arrow{r}{s^0} \arrow[equals]{d}[swap]{} & {[0]} \arrow[equals]{d}[swap]{} \\
		{[1]} \arrow{r}{s^0} & {[0]} 
	\end{tikzcd}\qquad\begin{tikzcd}
		{[1]} \arrow{r}{d^1} \arrow[equals]{d}[swap]{} & {[2]} \arrow[equals]{d}[swap]{} \\
		{[1]} \arrow{r}{d^1} & {[2]} 
	\end{tikzcd}\qquad\begin{tikzcd}
		{[1]} \arrow{r}{s^0} \arrow{d}[swap]{s^0} & {[0]} \arrow[equals]{d}[swap]{} \\
		{[0]} \arrow[equals]{r}{} & {[0]} 
	\end{tikzcd}\qquad\begin{tikzcd}
		{[1]} \arrow{r}{d^1} \arrow{d}[swap]{s^0} & {[2]} \arrow{d}[swap]{} \\
		{[0]} \arrow[equals]{r}{} & {[0]} 
	\end{tikzcd}
	\]
%	Furthermore:
%	\begin{enumerate}
%		\item Every basic pushout square of two coface maps is a $\vee$-product of only trivial pushout squares.
%		\item Every concrete basic pushout is a $\vee$-product of only trivial pushout squares and the pushout square above right.
%	\end{enumerate}
\end{cor}

\begin{proof}
	By \Cref{prop:all_basic_pushouts}, each pushout square factors into basic pushouts and trivial pushouts of generators, or equivalently pushouts of spans whose maps have defect 0 or 1. By \Cref{lem:star_defect_compatible}, if a $\vee$-product of pushout squares has this property then so do its $\vee$-components, which must then be among the squares of \Cref{thm:pushouts_delta_star} with this property. These are precisely the squares above.
\end{proof}

It is straightforward to check that the analogous generators for concrete pushouts include all but the bottom right square above.

\subsection*{Balanced coface squares}

Recall that every coface map in $\Delta$ is a composite of generating coface maps $d^i : {[n-1]} \to {[n]}$ for $0 \le i \le n$. The generating relations between these generators are given by simplicial identities of the following form:
\begin{equation}
	\label{basic_coface_squares}
	\begin{tikzcd}
		{[n-2]} \arrow{r}{d^i} \arrow{d}[swap]{d^{j-1}} & {[n-1]} \arrow{d}[swap]{d^j} \\
		{[n-1]}  \arrow{r}{d^i} & {[n]}	\\
		{} \ar[phantom]{r}{(0 \le i < j \le n)}		& {}
	\end{tikzcd}
\end{equation}
We call these squares (and their reflections) \emph{basic coface squares}. These are slightly more general than the squares in \Cref{type2face}, as the case $i=j-1$ is now included. While that square is not a pushout in $\Delta$, it becomes a pushout after applying the forgetful functor $\Delta \to \Set$. We show below that any square with this property factors into squares of the above form.

\begin{defn}
A commuting square of coface maps in $\Delta$ is \emph{balanced} if it is a pushout of finite sets.
\end{defn}

This terminology is motivated by the following characterization.

\begin{lemma}\label{lem:balanced_bicartesian}
A square of coface maps in $\Delta$ as below is balanced if and only if it is jointly surjective and $p + q = m + n$.
\begin{center}\begin{tikzcd}
	{[m]} \arrow{r}{f} \arrow{d}[swap]{g} & {[p]} \arrow{d}[swap]{h} \\
	{[q]} \arrow{r}{k} & {[n]} 
\end{tikzcd}\end{center}
\end{lemma}

\begin{proof}
	If the square is a pushout of finite sets with injections $f$ and $g$, then ${[n]} \cong {[p]} \cup_{[m]} {[q]}$, so $n = p + q - m$ and the square is jointly surjective.  If the square is jointly surjective and $p + q = m + n$, then by joint surjectivity the induced map ${[p]} \cup_{[m]} {[q]} \to {[n]}$ is a surjection, but as $n = p + q - m$, this is a surjection between finite sets of the same cardinality, hence an isomorphism.
\end{proof}

Note that a pushout of monomorphisms of sets is also a pullback. %, and a more general result for bicartesian squares in $\Delta$ without restricting to coface maps could be proven similarly, though it will not be relevant to this paper.
In terms of defects, the equation $p + q = m + n$ is equivalent to $\delta_f = \delta_k$ and also to $\delta_g = \delta_h$. 

\begin{thm}\label{prop:balanced_factorization}
	A nontrivial commuting square of coface maps in $\Delta$ is balanced if and only if it can be factored into a grid of basic coface squares.
%\[
%\begin{tikzcd}
%	{[n-2]} \arrow{r}{d^i} \arrow{d}[swap]{d^{j-1}} & {[n-1]} \arrow{d}[swap]{d^j} \\
%	{[n-1]}  \arrow{r}{d^i} & {[n]}	\\
%	{} \ar[phantom]{r}{(0 \le i < j \le n)}		& {}
%\end{tikzcd}
%\]
\end{thm}

\begin{proof}
	The ``only if'' direction is immediate as pushouts are closed under composition.

	For the ``if'' direction, we consider a balanced square as below and prove the existence of a factorization by induction on the total defect $\delta_f + \delta_g = \delta_h + \delta_k$.
\[
\begin{tikzcd}
	{[m]} \arrow{r}{f} \arrow{d}[swap]{g} & {[p]} \arrow{d}[swap]{h} \\
	{[q]} \arrow{r}{k} & {[n]} 
\end{tikzcd}
\]
The basic coface squares are precisely the balanced squares between maps with defect $1$. Therefore using induction and reflection symmetry, it suffices to show that if $\delta_f = \delta_k > 1$, then the square can be factored horizontally into two nontrivial balanced squares.

As $f$ is a nontrivial coface, there exists a choice of $i \in {[p]} \setminus \im(f)$. For such an element $i$, we can factor $f$ uniquely as $f' : {[m]} \to {[p-1]}$ followed by $d^i : {[p-1]} \to {[p]}$, as in the diagram below. The same is true for $k$ with respect to $h(i) \in {[n]}$, which is not in $\im(k)$ as the original square is a pullback and $h(i)$ is by assumption not in $\im(hf)$. We then have the following factorization into squares, where $h'$ is defined as the restriction of $h$ along $d^i$, ensuring that both squares commute.
\begin{center}\begin{tikzcd}
		{[m]} \arrow{r}{f'} \arrow{d}[swap]{g} & {[p-1]} \arrow{d}[swap]{h'} \ar{r}{d^i}	& {[p]} \arrow{d}[swap]{h} \\
		{[q]} \arrow{r}{k'} & {[n-1]} \arrow{r}{d^{h(i)}}	& {[n]}
\end{tikzcd}\end{center}

The objects in both squares clearly satisfy the size equation of \Cref{prop:balanced_factorization}, so to show they are balanced it remains only to show that they are jointly surjective. In the right square, the only element of $n$ not in the image of $d^{h(i)}$ is $h(i)$, which is definitionally in the image of $h$. In the left square, restricting $h$ along $d^i$ excludes only $h(i)$ from the joint image, but $h(i)$ is also excluded from ${[n-1]}$ (with respect to the factorization through $d^{h(i)}$), so joint surjectivity follows from that of the outer rectangle.
\end{proof}

\begin{cor}\label{balanced_decompose}
	Balanced squares in $\Delta$ are generated under $\vee$ and composition by the following squares and their mirror images:
	\[
	\begin{tikzcd}
		{[0]} \arrow[equals]{r}{} \arrow[equals]{d}[swap]{} & {[0]} \arrow[equals]{d}[swap]{} \\
		{[0]} \arrow[equals]{r}{} & {[0]} 
	\end{tikzcd}\qquad\begin{tikzcd}
		{[1]} \arrow[equals]{r}{} \arrow[equals]{d}[swap]{} & {[1]} \arrow[equals]{d}[swap]{} \\
		{[1]} \arrow[equals]{r}{} & {[1]} 
	\end{tikzcd}\qquad\begin{tikzcd}
		{[0]} \arrow{r}{d^0} \arrow[equals]{d}[swap]{} & {[1]} \arrow[equals]{d}[swap]{} \\
		{[0]} \arrow{r}{d^0} & {[1]} 
	\end{tikzcd}\qquad\begin{tikzcd}
		{[0]} \arrow{r}{d^1} \arrow[equals]{d}[swap]{} & {[1]} \arrow[equals]{d}[swap]{} \\
		{[0]} \arrow{r}{d^1} & {[1]} 
	\end{tikzcd}
	\]\[
	\begin{tikzcd}
		{[1]} \arrow{r}{d^1} \arrow[equals]{d}[swap]{} & {[2]} \arrow[equals]{d}[swap]{} \\
		{[1]} \arrow{r}{d^1} & {[2]} 
	\end{tikzcd}\qquad\begin{tikzcd}
		{[0]} \arrow{r}{d^0} \arrow{d}[swap]{d^0} & {[1]} \arrow{d}[swap]{d^0} \\
		{[1]} \arrow{r}{d^1} & {[2]} 
	\end{tikzcd}\qquad\begin{tikzcd}
		{[0]} \arrow{r}{d^1} \arrow{d}[swap]{d^1} & {[1]} \arrow{d}[swap]{d^1} \\
		{[1]} \arrow{r}{d^2} & {[2]} 
	\end{tikzcd}\qquad\begin{tikzcd}
		{[1]} \arrow{r}{d^1} \arrow{d}[swap]{d^1} & {[2]} \arrow{d}[swap]{d^1} \\
		{[2]} \arrow{r}{d^2} & {[3]} 
	\end{tikzcd}
	\]
\end{cor}

For this to make sense, we note that it is straightforward to check that $\vee$ preserves balanced squares using \Cref{lem:balanced_bicartesian}.

\begin{proof}
By \Cref{prop:balanced_factorization} any balanced square factors into basic squares or trivial squares on generating cofaces. This collection of squares is preserved by taking $\vee$-components, so it suffices to list the basic squares and trivial squares on generating cofaces which arise as canonical $\vee$-components, and these are precisely the squares above.
\end{proof}

\section{Completeness and exactness properties}
\label{sec:compl_exact}

In this section, we consider properties of the set of squares in $\Delta$ that a given simplicial set $X$ sends to pullbacks or weak pullbacks in $\Set$.

\subsection*{Completeness}

We say a simplicial set $X$ is \emph{complete} with respect to a given collection of squares in $\Delta$ if it sends those squares to weak pullback squares. This terminology is motivated by the following section on lifting properties, where completeness corresponds to the ability to ``complete'' a certain type of diagram in $X$ to a simplex according to the maps in the square.  For fixed $X$ we consider the largest collection of squares with this property.

\begin{defn}
For a simplicial set $X$, let $\Comp(X)$ denote the set of all squares in $\Delta$ sent to weak pullbacks by $X$.
\end{defn}

Many squares in $\Delta$ belong to $\Comp(X)$ for any $X$.

\begin{prop}\label{always_comp}
$\Comp(X)$ contains all trivial squares, pushouts of codegeneracy maps, and squares of the form below with $f$ a coface map.
\[
\begin{tikzcd}
	{[m]} \arrow[equals]{r} \arrow[equals]{d} & {[m]} \arrow{d}[swap]{f} \\
	{[m]} \arrow{r}{f} & {[n]}
\end{tikzcd}
\]
\end{prop}

\begin{proof}
Trivial squares are sent to trivial squares, which are always pullbacks. Pushouts of codegeneracies are absolute pushouts by \cite[Theorem 1.2.1]{JoyalTierneyNotes}, meaning any contravariant functor sends them to pullbacks. $X$ sends the square above to the following.
\[
\begin{tikzcd}
	X_n \arrow{r}{X_f} \arrow{d}[swap]{X_f} & X_m \arrow[equals]{d} \\
	X_m \arrow[equals]{r} & X_m
\end{tikzcd}
\]
The pullback of the cospan is the identity square on $X_m$, with induced map from $X_n$ necessarily given by the split epimorphism $X_f$, so the square is a weak pullback by \Cref{pullback_split_epi}.
\end{proof}

From weak pullbacks being closed under composition, we immediately get analogous properties of $\Comp(X)$.

\begin{prop}\label{compose_comp}
$\Comp(X)$ contains all trivial squares and is closed under horizontal and vertical compositions of squares in $\Delta$.
\end{prop}

We can now use the factorization results of the previous sections to characterize when $X$ is complete with respect to various collections of squares.

\begin{thm}
	\label{complete_characterization}
	Let $X$ be a simplicial set. Then:
\end{thm}
\hspace{-2.4cm}\begin{centering}
	\renewcommand{\arraystretch}{1.2} % some vertical padding
	\begin{tabular}{c|c}
		$\Comp(X)$ contains all & if and only if it contains the squares of \\
		\hline
		pushouts of coface maps & \Cref{type2face} \\
		pushouts of one coface and one degeneracy map & \Cref{typemixed} \\
		all pushouts & \Cref{type2face} and \Cref{typemixed} \\
		concrete pushouts & \Cref{type2face} and the left and right squares in \Cref{typemixed} \\
		balanced squares & basic coface squares
	\end{tabular}
\end{centering}
\medskip

\begin{proof}
The first two claims follow from \Cref{pure_pushouts} and \Cref{compose_comp}. The next two claims additionally rely on \Cref{always_comp} which removes the need to check for pushouts of codegeneracies. The final claim follows from \Cref{prop:balanced_factorization} and \Cref{compose_comp}.
\end{proof}

\begin{prop}\label{full_discrete}
If $\Comp(X)$ contains the squares of the form below for $0 \le i \le n$, then $X$ is discrete.
\[
	\begin{tikzcd}
		{[n+1]} \arrow{r}{d^{i+1}} \arrow{d}[swap]{s^i}	& {[n+2]} \arrow{d}{s^i s^i}	\\
		{[n]} \arrow[equals]{r}			& {[n]}			\\
	\end{tikzcd}\qquad\qquad\begin{tikzcd}
		{[n-2]} \arrow{r}{d^i} \arrow{d}[swap]{d^i} & {[n-1]} \arrow{d}[swap]{d^{i+1}} \\
		{[n-1]}  \arrow{r}{d^i} & {[n]}	\\
	\end{tikzcd}
\]
\end{prop}

\begin{proof}
Observe the squares below left compose to the square below right.
\[
	\begin{tikzcd}
		{[n]} \arrow{r}{d^i} \arrow{d}[swap]{d^i} & {[n+1]} \arrow{r}{s^i} \arrow{d}[swap]{d^{i+1}}	& {[n]} \arrow[equals]{d}	\\
		{[n+1]} \arrow{r}{d^i} & {[n+2]} \arrow{r}{s^i s^i}			& {[n]}			\\
	\end{tikzcd}\qquad\qquad\begin{tikzcd}
		{[n]} \arrow[equals]{r} \arrow{d}[swap]{d^i} & {[n]} \arrow[equals]{d} \\
		{[n+1]}  \arrow{r}{s^i} & {[n]}	\\
	\end{tikzcd}
\]
As both of the squares above left are in $\Comp(X)$, so is the square above right by \Cref{compose_comp}. Furthermore, as $s^i$ is a split epic $s_i : X_n \to X_{n+1}$ is a split monic, the right square is sent to a strong pullback by \Cref{mono_strong_pullback}. $s_i$ is then a pullback of the identity, and hence an isomorphism. It follows that all face and degeneracy maps of $X$ are isomorphisms, so that $X$ is discrete.
\end{proof}

In particular, this is the case if $\Comp(X)$ contains both pushouts and balanced squares.

\subsection*{Exactness}

Analogously to completeness, $X$ is \emph{exact} with respect to a collection of squares in $\Delta$ if it sends those squares to strong pullbacks.

\begin{defn}
For a simplicial set $X$, let $\Ex(X)$ denote the set of all squares in $\Delta$ sent to pullbacks by $X$.
\end{defn}

The following was shown in the proof of \Cref{always_comp}.

\begin{prop}\label{always_ex}
$\Ex(X)$ contains all trivial squares and pushouts of degeneracies.
\end{prop}

As any strong pullback is a weak pullback, $\Ex(X) \subseteq \Comp(X)$. Some squares have the property that if they belong to $\Comp(X)$ they must further belong to $\Ex(X)$.

\begin{prop}\label{comp_to_ex}
$\Ex(X)$ contains all squares in $\Comp(X)$ of the form below with either $h$ or $k$ a codegeneracy.
\[
\begin{tikzcd}
	{[m]} \ar{r}{} \ar{d}[swap]{}  & {[p]} \ar{d}[swap]{h} \\
	{[q]} \ar{r}{k} & {[n]}
\end{tikzcd}
\]
\end{prop}

\begin{proof}
If $h$ or $k$ is split epic, $X$ sends it to a split monic, so $X$ sends the square to one of the form in \Cref{mono_strong_pullback}.
\end{proof}

Closure of pullbacks under composition, along with the pullback lemma, give us the following.

\begin{prop}\label{compose_ex}
$\Ex(X)$ is closed under composition and left- or upper-cancellation of squares.
\end{prop}

Unlike $\Comp(X)$, we can give conditions for $\Ex(X)$ to be closed under $\vee$-products.

\begin{prop}
	\label{prop:star_product_pullbacks}
	If $\Ex(X)$ contains all squares of the following form, then it is closed under $\vee$-products.
\[
\begin{tikzcd}
	{[0]} \arrow{r}{a} \arrow{d}[swap]{0} & {[a]} \arrow{d}[swap]{} \\
	{[b]} \arrow{r}{}  & {[a+b]}
\end{tikzcd}
\]
\end{prop}

\begin{proof}
As $\vee$ is associative, %stated after definition of \vee
by induction it suffices to show this for binary $\vee$-products. Assume the squares above belong to $\Ex(X)$, as well as two generic squares which admit a $\vee$-product: 
\begin{center}
	\begin{tikzcd}
		{[m_0]} \arrow{r}{} \arrow{d}[swap]{} & {[p_0]} \arrow{d}[swap]{} \\
		{[q_0]} \arrow{r}{} & {[n_0]} 
	\end{tikzcd}\qquad\qquad
	\begin{tikzcd}
		{[m_1]} \arrow{r}{} \arrow{d}[swap]{} & {[p_1]} \arrow{d}[swap]{} \\
		{[q_1]} \arrow{r}{} & {[n_1]} 
	\end{tikzcd}
\end{center}
By assumption, the squares below left ($i=0,1$) and center are pullbacks, along with those like below center with $n$ replaced with $r,p,q$. Since limits commute with limits, the square below right is also a pullback.
\begin{center}
	\begin{tikzcd}
		X_{n_i} \arrow{r}{} \arrow{d}[swap]{} & X_{p_i} \arrow{d}[swap]{} \\
		X_{q_i} \arrow{r}{} & X_{r_i} 
	\end{tikzcd}\qquad\qquad
	\begin{tikzcd}
		X_{n_0 + n_1} \ar{r} \ar{d}		& X_{n_0} \ar{d} \\
		X_{n_1} \ar{r}			& X_0
	\end{tikzcd}\qquad\qquad
	\begin{tikzcd}
		X_{n_0+n_1} \arrow{r}{} \arrow{d}[swap]{} & X_{p_0+p_1} \arrow{d}[swap]{} \\
		X_{q_0+q_1} \arrow{r}{} & X_{r_0+r_1} 
	\end{tikzcd}
\end{center}
\end{proof}

This lets us use the results on generators under $\vee$ and composition to give much simpler conditions for a simplicial set $X$ to be exact with respect to various classes of squares.

\begin{thm}
	\label{exact_characterization}
	Let $X$ be a simplicial set.
	\begin{enumerate}
		\item $\Ex(X)$ contains all concrete pushouts if and only if it contains the squares below left.
		\item $\Ex(X)$ contains all pushouts if and only if it contains the squares below left and center left.
		\item $\Ex(X)$ contains all balanced squares if and only if it contains the squares below left, center right, and right.
	\end{enumerate}
\[
	\begin{tikzcd}
		{[0]} \arrow{r}{a} \arrow{d}[swap]{0} & {[a]} \arrow{d}[swap]{} \\
		{[b]} \arrow{r}{}  & {[a+b]}
	\end{tikzcd}\quad\begin{tikzcd}
		{[1]} \arrow{r}{d^1} \arrow{d}[swap]{s^0} & {[2]} \arrow{d}[swap]{} \\
		{[0]} \arrow[equals]{r}{} & {[0]} 
	\end{tikzcd}\quad\begin{tikzcd}
		{[0]} \arrow{r}{d^0} \arrow{d}[swap]{d^0} & {[1]} \arrow{d}[swap]{d^0} \\
		{[1]} \arrow{r}{d^1} & {[2]} 
	\end{tikzcd}\quad\begin{tikzcd}
		{[0]} \arrow{r}{d^1} \arrow{d}[swap]{d^1} & {[1]} \arrow{d}[swap]{d^1} \\
		{[1]} \arrow{r}{d^2} & {[2]} 
	\end{tikzcd}%\quad\begin{tikzcd}
		%{[1]} \arrow{r}{d^1} \arrow{d}[swap]{d^1} & {[2]} \arrow{d}[swap]{d^1} \\
		%{[2]} \arrow{r}{d^2} & {[3]} 
	%\end{tikzcd}
\]
\end{thm}

\begin{proof}
For the ``only if'' directions, note that for all $a,b$ the square above left is both a concrete pushout and balanced, the center left square is a pushout, and the left, center right, and right squares are balanced. We now consider the ``if'' direction.

$\Ex(X)$ is closed under composition, so if it is also closed under $\vee$ by containing the squares above left, it contains all (concrete) pushouts if and only if it contains the generators of (concrete) pushouts under $\vee$ and composition. By \Cref{lem:basic_decompose}, for concrete pushouts these generators are all trivial squares or pushouts of codegeneracies which are automatically in $\Ex(X)$, and for general pushouts they additionally contain the square above center left, proving the first two claims. 

For the third claim, it follows from \Cref{balanced_decompose} that $\Ex(X)$ contains all balanced squares if and only if it contains their generators under $\vee$ and composition, assuming it includes the squares above left. The only nontrivial generators are the center right and right squares above along with the following square, so further containing these suffices to show that $\Ex(X)$ includes all balanced squares.
\[
	\begin{tikzcd}
		{[1]} \arrow{r}{d^1} \arrow{d}[swap]{d^1} & {[2]} \arrow{d}[swap]{d^1} \\
		{[2]} \arrow{r}{d^2} & {[3]} 
	\end{tikzcd}
\]
We show that if $\Ex(X)$ contains the center right and right squares, it contains this one as well. Observe that the following two diagrams have the same composite square.
\[
\begin{tikzcd}
	{[0]} \arrow{r}{d^0} \arrow{d}[swap]{d^0} & {[1]} \arrow{r}{d^1} \arrow{d}[swap]{d^1} & {[2]} \arrow{d}[swap]{d^2} \\
	{[1]} \arrow{r}{d^0} & {[2]} \arrow{r}{d^1} & {[3]} 
\end{tikzcd}\qquad\qquad\begin{tikzcd}
	{[0]} \arrow{r}{d^0} \arrow{d}[swap]{d^0} & {[1]} \arrow{r}{d^0} \arrow{d}[swap]{d^1} & {[2]} \arrow{d}[swap]{d^2} \\
	{[1]} \arrow{r}{d^0} & {[2]} \arrow{r}{d^0} & {[3]} 
\end{tikzcd}
\]
The right square in the right composite is the $\vee$-product of the following trivial squares and therefore in $\Ex(X)$.
\[
\begin{tikzcd}
	{[1]} \arrow[equals]{r} \arrow{d}[swap]{d^1} & {[1]} \arrow{d}[swap]{d^1} \\
	{[2]} \arrow[equals]{r} & {[2]}
\end{tikzcd}
\qquad\qquad
\begin{tikzcd}
	{[0]} \arrow{r}{d^0} \arrow[equals]{d} & {[1]} \arrow[equals]{d} \\
	{[0]} \arrow{r}{d^0}  & {[1]}
\end{tikzcd}
\]
Therefore using \Cref{compose_ex}, the left composite above is in $\Ex(X)$ and by cancellation 
so is its right square, completing the proof.
\end{proof}

\begin{remark}
This shows that $\Ex(X)$ contains all concrete pushouts if and only if $X$ is the nerve of a category. Indeed, we have shown that $\Ex(X)$ contains all concrete pushouts precisely when $X_{a+b} \cong X_a \times_{X_0} X_b$ for all $a,b$, equivalent to having $X_n = X_1 \times_{X_0} \scdots{n} \times_{X_0} X_1$ for all $n$, the Segal condition for $X$ to be a nerve. 

Now let us consider the more specific situation where $\Ex(X)$ contains all pushouts, and ask what this means for the category of which $X$ is the nerve. $X$ sending the square above center left to a pullback is then equivalent to whenever a composable pair of morphisms (an element of $X_2$) has as composite arrow (in $X_1$) an identity (in the image of $s_0 : X_0 \to X_1$), both morphisms must be identities (in the image of $X_2 \to X_0$). In other words, $X$ is the nerve of a category in which no nontrivial morphisms compose to the identity. 

In the third case, when $\Ex(X)$ contains all balanced squares, the center right square above being sent to pullbacks ensures that every pair of morphisms $f : y \to z$, $g : x \to z$ in $X$ complete to a triangle with $h : x \to y$ and $fh=g$, so setting $g = \id_z$ provides any morphism $f$ with a right inverse $h$. Likewise the right square sent to a pullback provides each morphism with a left inverse, so that $X$ is the nerve of a groupoid.
\end{remark}

\begin{ex}\label{decomp_spaces}
Another class of simplicial sets defined by an exactness property is the (discrete special case of) decomposition spaces of \cite{GKT1} (or equivalently 2-Segal simplicial sets) which send the squares of \Cref{2segal_stiff_split_squares} to pullbacks. \cite[Proposition 3.5]{GKT1} shows that it suffices to check this for a smaller collection of squares, a result much like the cases considered in \Cref{exact_characterization}.
\end{ex}

\begin{ex}
A simplicial set $X$ for which $\Ex(X)$ contains the left and right squares of \Cref{typemixed} is a \emph{stiff} simplicial set in the sense of \cite[4.1]{GKT2}. That $X$ being stiff implies that all of the left and right squares of \cref{typemixed} are contained in $\Ex(X)$ is given by \cite[Lemma 4.3]{GKT2}, and that this condition implies all codegeneracy/inert pushouts are in $\Ex(X)$ follows from \Cref{pure_pushouts}, as these are among the concrete codegeneracy/coface pushouts which factor into these basic squares.
\end{ex}

\begin{ex}
If $\Ex(X)$ contains all of the squares of \Cref{typemixed}, it is \emph{split} in the sense of \cite[5.1]{GKT2}. 
$\Ex(X)$ containing the middle squares of \Cref{typemixed} means $X$ has \emph{indecomposable units} in the sense of \cite[5.5]{GKT2}, containing the left and right squares of \Cref{typemixed} makes $X$ stiff as described above, and any simplicial set $X$ is \emph{complete} in the sense of \cite[2.1]{GKT2} as degeneracy maps are always monomorphisms (unlike in the setting of simplicial spaces). 
Completeness in this sense also follows from \Cref{always_ex} following \cite[2.7]{GKT2} relating the condition to sending certain pushouts of degeneracies to pullbacks, which is always true in the discrete setting. 
By \cite[Proposition 5.9]{GKT2}, $X$ is split precisely when it is stiff, complete, and has indecomposable units, hence exactly when $\Ex(X)$ contains the squares of \Cref{typemixed}. By \Cref{pure_pushouts}, split simplicial sets can equivalently be defined as those which send to pullbacks all pushouts in $\Delta$ of one coface and one codegeneracy map.
\end{ex}

%\begin{remark}
%Noting that an inner exact simplicial set sends the squares above to pullbacks and is therefore 1-Segal, the converse that the 1-Segal condition implies stiffness is shown in \cite[Proposition 2.5.3]{DK2segal}. However, our proof in \Cref{IE_characterization} is distinct from theirs, using the $\vee$-decomposition of concrete pushout squares in $\Delta$ into trivial and absolute pushouts. This proof also differs from that of \cite{2segalunital}, which shows that the 2-Segal condition implies stiffness using the pullback lemma.
%\end{remark}

\section{Lifting conditions and acyclic configurations}

We now give an equivalent description of completeness properties via lifting conditions, and explore additional lifting properties of simplicial sets complete with respect to either pushouts or balanced squares of face maps. 

%\subsection*{Completeness as span filling}

Like many conditions considered for simplicial sets, such as the (inner) horn filling conditions defining Kan complexes (resp.~quasicategories), completeness with respect to a class of squares can be expressed in terms of lifting properties. For a square in $\Delta$ as below left and $X$ a simplicial set, the square below center is a weak (resp.~strong) pullback if and only if $X$ has (unique) lifts against the map from the pushout $\Delta^p {}_{f}\!\sqcup_{g} \Delta^q$ of simplicial sets to $\Delta^n$ induced by $h,k$. That is, any map from this pushout into $X$ extends (uniquely) to an $n$-simplex as in the lifting diagram below right.

\[
\begin{tikzcd}
	{[m]} \ar{r}{f} \ar{d}[swap]{g}  & {[p]} \ar{d}[swap]{h} \\
	{[q]} \ar{r}{k} & {[n]}
\end{tikzcd}\qquad\begin{tikzcd}
	X_n \arrow{r}{X_h} \arrow{d}[swap]{X_k} & X_p \arrow{d}[swap]{X_f} \\
	X_q \arrow{r}{X_g} & X_m
\end{tikzcd}\qquad\begin{tikzcd}
	\Delta^p \; {}_{f}\!\sqcup_{g} \Delta^q \arrow{r} \arrow{d}{k}[swap]{h} & X \\
	\Delta^n \arrow[dashed]{ur}
\end{tikzcd}
\]

This lets us interpret completeness properties geometrically, and we now describe several different completeness properties which admit simple geometric descriptions in terms of these filler conditions.

\begin{defn}
	\label{def:span_complete}
	A simplicial set $X$ is \emph{span complete} if $\Comp(X)$ includes all balanced squares of coface maps.
\end{defn}

We say a \emph{span (of cofaces)} in $X$ is a pair of simplices which share a face, which is equivalent to a span of coface maps in the category of elements for $X$. In a span complete simplicial set, any span in $X$ extends as above to a ``filler'' simplex between all of the vertices of the span.

By \Cref{complete_characterization}, %\ref{compl_char_balanced}, 
$X$ is span complete if and only if $\Comp(X)$ contains the basic coface squares, or equivalently any \emph{basic span} consisting of a pair of $(n-1)$-simplices in $X$ which share an $(n-2)$-simplex face can be filled to an $n$-simplex in $X$ (for all $n \ge 2$). That is, any $(n-1)$-simplices $x,y$ in $X$ with $d_i x = d_{j-1} y$, for $i < j$, extends to an $n$-simplex $z$ with $d_j z = x$ and $d_i z = y$. The following figures illustrate this for $n = 3$, with $i = 2$, $j = 3$ and $i = 0$, $j = 2$, respectively.
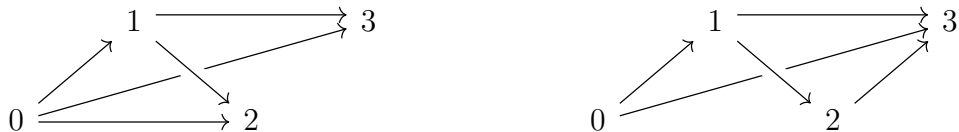
\[\begin{tikzcd}
& 1 \arrow[shift left=1]{rr} & & 3 \\
0 \arrow{ur} \arrow{rr} \arrow{urrr} & & 2 \arrow[from=1-2, crossing over]
\end{tikzcd}\qquad\qquad\qquad\begin{tikzcd}
& 1 \arrow[shift left=1]{rr} & & 3 \\
0 \arrow{ur} \arrow{urrr} & & 2 \arrow{ur} \arrow[from=1-2, crossing over]
\end{tikzcd}\]

However, asking for all basic span fillers rules out all nerves of categories but groupoids, as fillers against the basic span inclusions $\Delta^1 {}_{d^0}\!\sqcup_{d^0} \Delta^1 \hookrightarrow \Delta^2$ and $\Delta^1 {}_{d^1}\!\sqcup_{d^1} \Delta^1 \hookrightarrow \Delta_2$ require that for any morphism $a$, the diagrams below complete to commuting triangles, hence $a$ must have both left and right inverses.
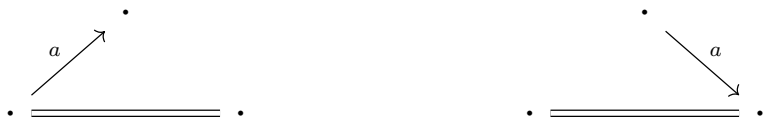
\[\begin{tikzcd}
& \LargerCdot \\
\LargerCdot \arrow{ur}{a} \arrow[equals]{rr} & & \LargerCdot
\end{tikzcd}\qquad\qquad\qquad\qquad\begin{tikzcd}
& \LargerCdot \arrow{dr}{a}\\
\LargerCdot \arrow[equals]{rr} & & \LargerCdot
\end{tikzcd}\]

A generalization of span complete simplicial sets which allows for any nerve of a category (as well as quasicategories, see \Cref{qcat_innerspans}) restricts the desired fillers to \emph{inner spans} analogous to the restriction of horns to inner horns when generalizing Kan complexes to quasicategories. Where the general spans above contain all vertices of the desired $n$-simplex, an inner span further contains all of the spinal edges, such as in the right span within the 3-simplex pictured above. In the nerve of a category then, this means that all of the morphisms in an $n$-simplex are provided by the span, and filling it to an $n$-simplex amounts to simply adding in the missing composites.  By \Cref{face_pushout_concrete}, these spans are precisely those arising from a pushout square of coface maps in $\Delta$, which motivates the following definition.

\begin{defn}
	\label{def:inner_span_complete}
	A simplicial set $X$ is \emph{inner span complete} if $\Comp(X)$ includes all pushout squares of coface maps.
\end{defn}

By \Cref{complete_characterization}, %\ref{compl_char_coface}, 
$X$ is inner span complete if and only if $\Comp(X)$ contains the basic pushout squares of face maps in \Cref{type2face}. Equivalently, $X$ is inner span complete precisely when any \emph{basic inner span} consisting of a pair of $(n-1)$-simplices in $X$ which share an $(n-2)$-simplex face and together include a string of $n$ successive edges can be filled to an $n$-simplex in $X$ (for all $n \ge 2$).
Such a span contains all but one edge of the desired $n$-simplex, and the innerness condition requires that this edge is not in the spine. 

Inner span complete simplicial sets describe settings in which edges can be composed in a manner respected by higher simplices but without requiring the uniqueness or coherence properties of categories and quasicategories, respectively. In \cite{weakly_cartesian} we show that this precisely describes the compositional structure possessed by the bar construction of algebras of a broad class of monads.
This generalizes the composition of \emph{partial evaluations}~\cite{FP} to higher simplices. By definition, a partial evaluation is an edge in the bar construction of an algebra of a monad. We have found that the compositional structure of partial evaluations established in~\cite{FP} for a large class of monads extends to the higher-dimensional simplices: the bar construction is an inner span complete simplicial set.
Potential applications to algebraic rewriting theory remain to be explored.

%Finally, we discuss a further weakening of these properties which additionally allows for the 2-Segal simplicial sets introduced in \cite{DK2segal}.
%
%\begin{defn}
%A simplicial set $X$ is \emph{weakly 2-Segal} if $\Comp(X)$ includes all pushout squares of coface maps as below for which $f,g$ either both preserve the maximal element or both preserve the minimal element. 
%\[
%\begin{tikzcd}
%	{[m]} \ar{r}{f} \ar{d}[swap]{g}  & {[p]} \ar{d}[swap]{h} \\
%	{[q]} \ar{r}{k} & {[n]}
%\end{tikzcd}
%\]
%\end{defn}
%
%
%
%
%
%For simplicial sets with any of these three completeness conditions (span and inner span completeness, weakly 2-Segal), the following subsections describe a broad class of additional lifting properties satisfied automatically.
%We start with the undirected case of simplicial complexes, where these conditions lifting properties have arisen historically in the context of relational database design.

\subsection*{Combinatorial acyclicity}

In this subsection, we prove the existence of certain additional fillers in a span complete simplicial set $X$, obtained by iterating the filling condition for basic spans.

For the moment we will work in the undirected context, considering \emph{(abstract) simplicial complexes}, in the standard sense of collection of subsets of a finite ground set which are downward closed, and such that the union of all these subsets is the ground set.
We will treat nonempty simplicial complexes as subsimplicial sets of the representable simplex on the ground set, where for now the order of the vertices will not matter (although it is specified). In the directed context that follows afterwards, we discuss how to modify these definitions to account for directed edges.

A vertex in a simplicial complex is \emph{extremal} if it is contained in only one maximal simplex. A \emph{combinatorial sphere} is a simplicial complex with at least 3 vertices, containing precisely the proper subsets of the ground set. One can visualise it geometrically as a hollow triangle, or a hollow tetrahedron, or in general the boundary of a simplex.

The following definitions and the characterization of \Cref{undirected_acyclic} are well-known, but they do not seem to be easy to find in the literature in this exact form. Much of the related literature is in the area of relational database theory, where often more general hypergraphs rather than simplicial complexes are considered\footnote{See e.g.~\cite{acyclic_desirability}, or~\cite[Chapter~13]{maier} for a textbook account.}, resulting in greater generality and complexity than what we need here.

\begin{defn}[{\cite{graham}}]
	A simplicial complex $S$ is \emph{Graham acyclic} if it satisfies the following recursive definition: $S$ is empty, or $S$ contains an extremal vertex $v$ and $S \setminus \{v\}$ is Graham acyclic. 
\end{defn}

Here, $S \setminus \{v\} \coloneqq \{A\setminus\{v\} \mid A \in S\}$ denotes the new simplicial complex obtained by removing $v$ from all simplices as well as from the ground set. Applying this recursive elimination of vertices to a simplicial complex is called \emph{Graham reduction}. For a Graham acyclic simplicial complex, the reduction results in the empty complex\footnote{It is known that Graham reduction can be performed in any order, i.e.~it is impossible to get stuck.}, while otherwise the process terminates at a non-empty complex.

Note that there are similarities with the notions of collapsibility and shellability in combinatorial topology. In particular, Graham acyclicity is by definition equivalent to Wegner's \emph{1-collapsibility}~\cite{wegner}. Also the following notion is standard, see e.g.~\cite[Definition~5.3.15]{west}.

% T: there shouldn't be any need to do relabellings in the undirected case, so for now I've removed the following - OK, I've copied it to the directed part
%
% Note that if the vertices of $S$ are labeled from $1$ to $n$ then the vertices of $S\setminus\{v\}$ are labeled from $1$ to $n-1$. This is achieved by subtracting one from all labels of vertices higher than $v$ in the ordering.

\begin{defn}
	A \emph{chordal graph} is an undirected graph in which all cycles with at least $4$ edges have a chord, i.e.~an edge which connects two vertices non-adjacent in the cycle.
\end{defn}

Applying this definition repeatedly shows that in a chordal graph, a cycle of any length can be triangulated, which is why chordal graphs are also sometimes called \emph{triangulated graphs}.

The following characterization theorem is well-known in its hypergraph version in the literature on acyclic database schemes, see e.g.~\cite{acyclic_desirability} or~\cite[Theorem~13.2]{maier}, while the proof is somewhat simpler in our setting of simplicial complexes. The condition~\ref{graham} is easy to check algorithmically, while condition~\ref{fillers} is useful for mathematical proofs. Condition \ref{rip} is the one which will facilitate our reduction to inner span fillers in the directed case below, and is generally useful when working with algebraic or combinatorial structures on simplicial complexes, such as the tables in a relational database.

\begin{thm}
	\label{undirected_acyclic}
	The following are equivalent for a simplicial complex $S$:
	\begin{enumerate}
		\item\label{graham} $S$ is Graham acyclic.
		\item\label{fillers} Every combinatorial sphere in $S$ has a filler, and the $1$-skeleton of $S$ is a chordal graph.
		%\item\label{vertexorder} There is a total ordering of the vertices as $v_1,\ldots,v_n$ such that for every $k = 1,\ldots,n$, if $A, B \in S$ with $k \in A,B$, then
		%	\[
	%			A \cap \{k,\ldots,n\} \subseteq B \cap \{k,\ldots,n\} \qquad\lor\qquad B \cap \{k,\ldots,n\} \subseteq A \cap \{k,\ldots,n\}.
	%		\]
		\item\label{rip} $S$ has the \emph{running intersection property}: the maximal simplices of $S$ can be ordered as $T_1, \ldots, T_m$ such that for every $k = 1,\ldots,m$ there is $j < k$ with
			\[
				T_k \cap \left( \bigcup_{i=1}^{k-1} T_i \right) \: \subseteq \: T_j.
			\]
	\end{enumerate}
	Moreover if $S$ is connected, then the $T_1,\ldots,T_m$ in \ref{rip} can be chosen such that $T_k \cap \bigcup_{i=1}^{k-1} T_i$ is nonempty for all $k = 2,\ldots,m$.
\end{thm}

We include a proof for convenience.

\begin{proof}
$\ref{graham}\Rightarrow \ref{fillers}$: We use induction on the number of vertices of $S$, with the statement being trivial if $S$ is empty. For the induction step, suppose that $S$ is Graham acyclic with extremal vertex $v$.

Now consider a combinatorial sphere in $S$. If this sphere does not contain $v$, then it has a filler by the induction assumption applied to $S \setminus \{v\}$. If this sphere contains $v$, then it must also have a filler, since otherwise $v$ would be contained in more than one maximal simplex.

Similarly, consider a cycle of length $\ge 4$ in the $1$-skeleton of $S$. If $v$ is not part of this cycle, then it again has a chord by the induction assumption, so suppose that $v$ is a vertex in the cycle. Then both neighboring vertices of $v$ in the cycle are also members of the unique maximal simplex containing $v$, and therefore so is the edge between these vertices, resulting in a chord.

$\ref{fillers}\Rightarrow \ref{graham}$: A vertex $v$ in a graph $G$ is called \emph{simplicial} if every two neighboring vertices of $v$ are themselves adjacent. A standard graph-theoretic result is that a graph is chordal if and only if there is an ordering $\{v_1, \ldots, v_n\}$ of its vertices such that each $v_i$ is simplicial in the subgraph induced by the vertices $\{v_1,\ldots,v_i\}$~\cite[Theorem~5.3.17]{west}. In the case of a simplicial complex whose 1-skeleton is a chordal graph, as long as there are no unfilled combinatorial spheres, such an ordering can be reversed to provide an ordering for the Graham reduction process. This is because every complete subgraph of the 1-skeleton has to be a simplex in $S$ by assumption, in particular making $v_n$ extremal in $S$.

$\ref{rip}\Rightarrow \ref{graham}$: If the running intersection property holds, then putting $k=m$ shows that there is a $j<m$ such that $T_m \cap \left( \bigcup_{i=1}^{m-1} T_i \right) \: \subseteq \: T_j$. This implies that there is some vertex $v\in T_m$ which does not belong to any of the other maximal simplices from $1$ to $m-1$, making $v$ extremal. Considering the reduced complex $S \setminus \{v\}$, there are now two possibilities: it may be that $S\setminus\{v\}$ has maximal simplices $T_1,\ldots T_{m-1}$ as maximal simplices, in which case the running intersection property still holds trivially; or $S \setminus \{v\}$ may in addition have the maximal simplex $T_m\setminus \{v\}$, in which case the running intersection property still holds with $T_m$ replaced by $T_m\setminus\{v\}$ in the new ordering. In either case, the induction assumption finishes the argument, again with the empty simplicial complex as the base case.

$\ref{graham}\Rightarrow \ref{rip}$: We once more use induction on the number of vertices, where the empty base case is obvious.
For the induction step, suppose that $v$ is an extremal vertex in $S$, belonging to a unique maximal simplex $T$. Since $S \setminus \{v\}$ is still Graham acyclic, the induction hypothesis shows that there exists an ordering $T_1,\ldots T_{m-1}$ of the maximal simplices of $S \setminus \{v\}$ which satisfies the running intersection property.

Then we again have two cases. First, if $T \setminus \{v\}$ is still maximal in $S \setminus \{v\}$, then it must coincide with some $T_k$. Then $T_1, \ldots, T_k \cup \{v\}, \ldots, T_m$ is an ordering of the maximal simplices of $S$ which witnesses the running intersection property. Second, if $T \setminus \{v\}$ is no longer maximal in $S \setminus \{v\}$, then it must be properly contained in some $T_j$. Then the sequence of maximal simplices
\[
	T_1, \ldots, T_m, T 
\]
witnesses the running intersection property for $S$, because of $T \cap \bigcup_{i=1}^m T_i = T \setminus \{v\} \subseteq T_j$.

Moreover, the final claim on connectedness follows by an inspection of the previous argument: if $S$ is connected, then so is $S \setminus \{v\}$, and it is straightforward to check that every $T_k \cap \bigcup_{i=1}^{k-1} T_i$ is nonempty provided that this holds likewise on $S \setminus \{v\}$, which it does by the induction assumption.
\end{proof}

\begin{remark}
	It should be noted that the acyclicity property characterized by \Cref{undirected_acyclic} is not homotopy invariant, and in particular distinct from notions of acyclicity familiar from algebraic topology. This applies similarly to our directed analogue below.
\end{remark}

\begin{defn}
	\label{acyclic_config}
	An \emph{acyclic configuration} inside the $n$-simplex is a connected simplicial complex $S$ with ground set $[n]$ which satisfies the conditions of~\Cref{undirected_acyclic}.
\end{defn}

\begin{ex}\label{span_acyclic}
	A span of simplices is an acyclic configuration inside their union. Indeed, any combinatorial sphere is the boundary of a face of one of the two simplices and hence has a filler. The 1-skeleton is the union of two complete graphs along another complete graph. In particular, every cycle has a chord: while this is obvious for a cycle contained in one of the complete subgraphs, a cycle not contained in either needs to have at least two vertices in the intersection, where again a chord exists.
\end{ex}

Acyclic configurations have the following relevance in our context.

\begin{thm}\label{acyclic_fillers}
	A simplicial set is span complete if and only if it has fillers for all acyclic configurations in the $n$-simplex (for every $n$).
\end{thm}

\begin{proof}
	The ``if'' direction follows from the acyclicity of spans in the example above. %, which is the undirected analogue of~\Cref{inner_span_acyclic}. The ``only if'' direction follows by a completely analogous argument to~\Cref{acyclicfiller} in the undirected setting, based on the observation that every pullback of coface maps in $\Delta$ is a balanced square and therefore sent to a weak pullback by the span complete simplicial set under consideration. The connectedness is relevant for ensuring that $T_k \cap \bigcup_{i=1}^{k-1} T_i$ is nonempty for every $k$, which indeed holds by the final statement in \Cref{undirected_acyclic}.
	For the ``only if'' direction, we use condition~\ref{rip} of \Cref{undirected_acyclic} characterizing acyclic configurations. Let $X$ be span complete and consider a map to $X$ from acyclic $S \subseteq \Delta^n$ with maximal simplices $\{T_1,T_2,...\}$. We show by induction on $k$ that every induced subcomplex on vertices $\bigcup_{i=1}^k T_i$ has a filler. There is nothing to prove in the base case $k = 1$, so assume $k > 1$. Then the induced subcomplex on vertices $\bigcup_{i=1}^{k-1} T_i$ has a simplex filler by the induction assumption. But now with the inclusion maps as morphisms, the diagram\footnote{Note that the assumption of connectedness guarantees that the set-theoretic intersection $T_k \cap \bigcup_{i=1}^{k-1}$ is nonempty.}
	\[
		\begin{tikzcd}
			T_k \cap \bigcup_{i=1}^{k-1} T_i \arrow{r} \arrow{d}	& T_k \arrow{d}		\\
			\bigcup_{i=1}^{k-1} T_i \arrow{r}			& \bigcup_{i=1}^k T_i
		\end{tikzcd}
	\]
	is a square of coface maps in $\Delta$ and evidently a pushout of finite sets.  Since $T_k \cap \bigcup_{i=1}^{k-1} T_i \subseteq T_j$ for some $j$, we know that the filler of $\bigcup_{i=1}^{k-1} T_i$ agrees with $T_k$ on the face with vertices those of $T_k \cap \bigcup_{i=1}^{k-1} T_i$, as both must agree with the corresponding face of $T_j$. These simplices therefore form a span and have a filler in $X$ which extends the restriction of the original map to $\bigcup_{i=1}^{k} T_i$, completing the induction step.
\end{proof}

\subsection*{Directed acyclicity}

We now describe corresponding collections of fillers for inner span complete simplicial sets, obtained by iterating the filler condition for inner spans. To this end, we propose a version of the above acyclicity notion in a directed setting which accounts for the orientations of edges and triangles, respectively. A \emph{directed simplicial complex} $S$ is a downward closed collection of subsets of a finite nonempty totally ordered set, which without loss of generality we take to be given by
\[
	{[n]} = \{0, \ldots, n\} = \bigcup S,
\]
thereby identifying a directed simplicial complex on $n$ vertices with a simplicial subcomplex of the $n$-simplex. As before we write $S \setminus \{v\} = \{ A \setminus \{v\} \mid A \in S \}$, where now this reduced directed simplicial set lives on ${[n-1]}$, so that the indices of all vertices beyond $v$ must be reduced by $1$.

All notions for which we do not introduced directed or 2-directed versions, such as extremality of a vertex, are used as in the undirected setting above.

\begin{defn}\label{directedgrahamacyclic}
	A directed simplicial complex $S \subseteq 2^{[n]}$ is \emph{directed Graham acyclic} if $n = 0$, or if $S$ has an extremal vertex $v \in {[n]}$ such that:
	\begin{enumerate}
		\item If $v > 0$, then $\{v-1,v\} \in S$.
		\item If $v < n$, then $\{v,v+1\} \in S$.
		\item $S \setminus \{v\}$ is again directed Graham acyclic.
	\end{enumerate}
\end{defn}

We then have a characterization analogous to that of \Cref{undirected_acyclic}.

\begin{thm}\label{directed_acyclic}
	The following are equivalent for a directed simplicial complex $S \subseteq 2^{\{0,\ldots,n\}}$ with $\bigcup S = \{0,\ldots,n\}$:
	\begin{enumerate}
		\item\label{directed_graham} $S$ is directed Graham acyclic.
		\item\label{directed_fillers} Every combinatorial sphere in $S$ has a filler, and the $1$-skeleton of $S$ is a chordal graph which contains the entire spine.
		\item\label{directed_rip} $S$ has the \emph{directed running intersection property}: the maximal simplices of $S$ can be ordered as $T_1, \ldots, T_m$ such that for every $k = 1,\ldots,m$ there is $j < k$ with
			\[
				\left( \bigcup_{i=1}^{k-1} T_i \right) \cap T_k \: \subseteq \: T_j,
			\]
			and for every two vertices $v < w$ which are consecutive in $\bigcup_{i=1}^k T_i$, we have $\{v,w\} \subseteq T_k$ or $\{v,w\} \subseteq \bigcup_{i=1}^{k-1} T_i$.
	\end{enumerate}
\end{thm}

Note that each one of these conditions implies its undirected counterpart given in \Cref{undirected_acyclic}.

\begin{proof}
	It is enough to show that the additional conditions relative to \Cref{undirected_acyclic} imply each other, assuming that the underlying undirected simplicial complex of $S$ is acyclic.

	Assuming \ref{directed_graham}, a simple induction argument indeed shows that $S$ contains the whole spine. For if $v \in {[n]}$ is as in \Cref{directedgrahamacyclic}, then $S \setminus \{v\}$ can be assumed to contains its entire spine by the induction assumption, and the extra condition on $v$ then implies that $S$ also contains the additional spinal edges not implied by those of $S \setminus \{v\}$. Conversely if \ref{directed_fillers} holds, then the conditions $\{v-1,v\} \in S$ for $v > 0$ and $\{v,v+1\} \in S$ for $v < n$ are part of the assumption that $S$ contains the entire spine.

	For the equivalence between \ref{directed_fillers} and \ref{directed_rip}, it is now enough to prove that the extra condition in \ref{directed_rip} is equivalent to $S$ containing the entire spine, provided that the undirected acyclicity of \Cref{undirected_acyclic} holds. Thus if \ref{directed_rip} holds, we now argue that $\{v,v+1\} \in S$ for every $v < n$. To this end, consider the smallest $k$ with $\{v,v+1\} \subseteq \bigcup_{i=1}^k T_i$. Then the assumption implies the desired $\{v,v+1\} \subseteq T_k$, since $\{v,v+1\} \subseteq \bigcup_{i=1}^{k-1} T_i$ would contradict the minimality of $k$.

	In the other direction, suppose that $S$ satisfies the undirected running intersection property and contains the entire spine. Let $v < w$ be two vertices consecutive in $\bigcup_{i=1}^k T_i$. We will use backwards induction on $k$ to prove the desired property $\{v,w\} \subseteq T_k$ or $\{v,w\} \subseteq \bigcup_{i=1}^{k-1} T_i$, or equivalently that the induced subcomplex on $\bigcup_{i=1}^k T_i$ contains its entire spine. This is clear in the base case $k = m$: for then we must have $w = v + 1$, so that the containing the entire spine assumption applies.

	For $k < m$, suppose first that $v$ and $w$ are still consecutive in $\bigcup_{i=1}^{k+1} T_i$. Since the induced subcomplex on $\bigcup_{i=1}^{k+1} T_i$ contains the entire spine by the induction assumption, we must have $\{v,w\} \subseteq T_h$ for some $h \le k + 1$. For $h \le k$ we are done, so assume $h = k + 1$. Then the running intersection property implies that there is $j \le k$ with $T_{k+1} \cap \bigcup_{i=1}^{k} T_i \subseteq T_j$. We therefore also conclude that $\{v,w\} \subseteq T_j$, which is enough.

	Finally if $v$ and $w$ are no longer consecutive in $\bigcup_{i=1}^{k+1} T_i$, then there are nonzero many elements $u_1,\ldots,u_\ell \in T_{k+1} \setminus \bigcup_{i=1}^k T_i$ such that the sequence
	\[
		v, u_1, \ldots, u_\ell, w
	\]
	consists of consecutive vertices in $\bigcup_{i=1}^{k+1} T_i$. The induction assumption together with $u_1,...,u_\ell \not\in \bigcup_{i=1}^k T_i$ then gives us that the edge formed by any two consecutive vertices in this list is in $T_{k+1}$. But then also $\{v,u_1,\ldots,u_\ell,w\} \subseteq T_{k+1}$, and in particular $\{v,w\} \subseteq T_{k+1} \cap \bigcup_{i=1}^k T_i$. But then again the running intersection property implies $\{v,w\} \subseteq T_j$ for some $j \le k$, as was to be shown.
\end{proof}

\begin{defn}
	\label{directed_acyclic_config}
	A \emph{directed acyclic configuration} inside the $n$-simplex is a directed simplicial complex $S$ on $[n]$ satisfying the conditions of~\Cref{directed_acyclic}.
\end{defn}

Note that the connectivity requirement which we had made in the undirected case (\Cref{acyclic_config}) is now automatic by inclusion of the spinal edges.

\begin{ex}\label{inner_span_acyclic}
	All inner spans define directed acyclic configurations, as follows. Suppose that the diagram below is a pushout of coface maps in $\Delta$.
	\begin{center}
	\begin{tikzcd}
		{[m]} \arrow{r}{f} \arrow{d}[swap]{g} & {[p]} \arrow{d}[swap]{h} \\
		{[q]} \arrow{r}{k} & {[n]} 
	\end{tikzcd}
	\end{center}
	Then consider the directed simplicial complex on ${[n]}$ given by
	\[
		S := \{ A \subseteq {[n]} \mid A \subseteq \im(h) \:\lor\: A \subseteq \im(k) \}.
	\]
	This $S$ is a directed acyclic configuration: the underlying undirected complex of $S$ is a union of two simplices glued along a common face, and is therefore (undirected) acyclic by \Cref{span_acyclic}. Since it moreover contains the entire spine by \Cref{face_pushout_concrete}, directed acyclicity follows.
	
	In particular, all basic inner span inclusions define directed acyclic configurations: for $0 \le i < j - 1\le n - 1$, the directed simplicial complex
	\[
		S := \{ A \subseteq {[n]} \mid i \not\in A \:\lor\: j \not\in A \}.
	\]
	is directed acyclic. Using the same $S$ with $i = j - 1$ would not work, since then the spine condition would be violated due to the spinal edge $\{j-1, j\}$ not being a member of $S$.
\end{ex}

The relevance of directed acyclicity in our context is the following general result.

\begin{thm}\label{acyclicfiller}
	A simplicial set is inner span complete if and only if it has fillers for all directed acyclic configurations in the $n$-simplex (for every $n$).
\end{thm}

\begin{proof} 
	The proof is a straightforward modification of the proof of \Cref{acyclic_fillers}, as inner spans are directed acyclic as shown above, and in the induction argument, if $S$ is assumed directed acyclic then the condition on consecutive vertices in \ref{directed_rip} ensures that the square
	\[
		\begin{tikzcd}
			T_k \cap \bigcup_{i=1}^{k-1} T_i \arrow{r} \arrow{d}	& T_k \arrow{d}		\\
			\bigcup_{i=1}^{k-1} T_i \arrow{r}			& \bigcup_{i=1}^k T_i
		\end{tikzcd}
	\]
	is a pushout in $\Delta$ by \Cref{face_pushout_concrete}.
%\footnote{Note that the consecutive vertices condition in the directed running intersection property guarantees that the set-theoretic intersection $T_k \cap \bigcup_{i=1}^{k-1}$ is nonempty.}
\end{proof}

\begin{ex}\label{1segal_acyclic}
	Consider the spine inclusions of the edges $0 \to 1 \to \cdots \to n$ into $\Delta^n$ for $n \ge 1$. These define a directed simplicial complex with the maximal simplices given by
	\[
		T_1 = \{0,1\}, \quad \ldots, \quad T_n = \{n-1, n\}.
	\]
	Since this directed simplicial complex has no cycles or combinatorial spheres, and trivially contains the entire spine, it defines a directed acyclic configuration.
	\Cref{acyclicfiller} thus implies that in an inner span complete simplicial set, every string of $n$ edges is the spine of an $n$-simplex. This is a weak version of the $1$-Segal condition in which the fillers are not required to be unique. Unlike the ordinary 1-Segal condition, in this case these fillers alone do not imply the existence of fillers for inner spans.
\end{ex}

\begin{ex}\label{2segal_acyclic}
	Consider any triangulation of the $(n+1)$-gon for $n \ge 2$, with vertices labeled in order from $0$ to $n$ as in the two examples below for $n=3$.
$$
	\begin{tikzcd}
		1 \ar{r} \ar{dr}	& 2 \ar{d}	\\
		0 \ar{r} \ar{u}	& 3				
	\end{tikzcd}\qquad\qquad	\begin{tikzcd}
		1 \ar{r} 					& 2 \ar{d}	\\
		0 \ar{r} \ar{u} \ar{ur}	& 3				
	\end{tikzcd}
$$
	Then the edges and triangles of the triangulation define a directed acyclic configuration in the $n$-simplex. Indeed the configuration contains the spine of the $n$-simplex, as the spinal edges are among the outer edges of the $n$-gon. As a triangulation, the $1$-skeleton of this configuration is a chordal graph, and the only combinatorial spheres are the filled triangles. By \Cref{directed_acyclic}\ref{directed_fillers}, this is a directed acyclic configuration.

	\Cref{acyclicfiller} thus shows that inner span complete simplicial sets also satisfy a weak version of the Dyckerhoff-Kapranov formulation of the 2-Segal property in terms of polygon triangulations (\cite[Definition 2.3.1]{DK2segal}), and that this is implied by the weak analogue of the corresponding exactness conditions in \cite{GKT1} and \cite{Walde}. In contrast to the strong case, the converse does not since non-unique fillers against 2-dimensional triangulations do not provide a way to fill inner spans of simplices of dimension greater than 2.
\end{ex}

\begin{ex}\label{dsegal_acyclic}
	Any triangulation of a polytope on vertices $0,\ldots,n$ has fillers of combinatorial spheres, and forms a directed acyclic configuration of the $n$-simplex if its $1$-skeleton is a chordal graph containing the spinal edges. For example, the cyclic polytope on vertices $0,1,\ldots,n$ in $d$-dimensional space (with $n \ge d$) is the convex hull of any $n+1$ points on the moment curve $t \mapsto (t,t^2,\ldots,t^d)$ for $t \in \R$~\cite{GaleCyclic}. Any triangulation of such a polytope automatically contains the spinal edges, so to show these triangulations are directed acyclic it remains only to show that their 1-skeletons are chordal.

When $d > 3$, by~\cite[Theorem 1]{GaleCyclic} the $1$-skeleton of each cyclic polytope is a complete graph, hence chordal, so these triangulations form directed acyclic configurations.  When $d=3$, for example by Gale's evenness criterion~\cite[Theorem~3]{GaleCyclic}, the $1$-skeleton of the cyclic polytope on $0,\ldots,n$ consists precisely of the spinal edges along with edges from $0$ to any vertex and from any vertex to $n$. As the only edges between vertices other than $0$ and $n$ are the spinal edges from $i$ to $i+1$, any cycle must then contain $0$ or $n$, which has an edge to every vertex in the cycle, so the $1$-skeleton is chordal, and therefore any triangulation of the polytope by $3$-simplices forms a directed acyclic configuration. 

Dyckerhoff and Kapranov suggest in \cite{DK2segal} that a ``$d$-Segal condition'' could be defined for any $d \ge 1$ as a simplicial set having unique fillers against the inclusion into $\Delta^n$ of any $d$-simplex triangulation of the $d$-dimensional cyclic polytope on $0,\ldots,n$. This definition is made precise by Poguntke (\cite[Definition 2.2]{Poguntke}), who restricts to just the ``upper'' and ``lower'' triangulations of the cyclic polytopes (fillers against just those two triangulations for each cyclic polytope suffice to provide fillers for all triangulations, but only if the fillers are unique, which is easy to see when $d=2$).  

The fact that cyclic polytope triangulations form directed acyclic configurations shows, thanks to \Cref{acyclicfiller}, that inner span completeness subsumes a weak version of the triangulation-style $d$-Segal condition for all $d$. In fact, this requires only that a simplicial set $X$ is complete with respect to the pushouts of coface squares in \Cref{2segal_stiff_split_squares} excluding basic pushouts of the first and last coface maps. It is possible to formulate another even stronger notion of acyclicity corresponding to lifting properties which follow from this weak analogue of the 2-Segal condition, but that is beyond the scope of this paper.

In \cite{Walde}, Walde formulates equivalent characterizations of the $d$-Segal conditions in terms of exactness with respect to higher dimensional cube diagrams in $\Delta$, recovering the appropriate squares of \cite{GKT1} for $d=1,2$. It is possible that weak versions of these conditions are also implied by inner span completeness, but this too is beyond our current scope.
\end{ex}

\section{Examples of (inner) span completeness}
\label{sec:examples}

We now discuss several examples of span complete and inner span complete simplicial sets. We quickly show that any any quasicategory is inner span complete and note that any Kan complex is span complete, and give several examples of span complete simplicial sets which are not quasicategories.  The motivating example for the development of inner span complete simplicial sets is the bar construction of algebras for certain types of monads, which we discuss in a follow-up paper \cite{weakly_cartesian}.
Each of these examples can now benefit from \Cref{acyclic_fillers} or \Cref{acyclicfiller}, with simple (directed) acyclicity conditions describing a broad class of configurations which have $n$-simplex fillers. %analogous to the classical conditions for a collection of database schema to be joinable into a single table (see \Cref{databases}).

The examples we discuss can also be found in \cite{FF} as examples of \emph{compositories}:

\begin{defn}[{\cite[Definition 2.2.2]{FF}}]
A compository is a simplicial set $X$ for which inner spans of the type
	\begin{equation}
		\label{compository_filler}
		\begin{tikzcd}
			\Delta^p \; {}_{r'}\! \sqcup_{\ell'} \Delta^q \arrow{r} \arrow[hook]{d}{r}[swap]{\ell,} & X \\
			\Delta^n \arrow[dashed]{ur}
		\end{tikzcd}
	\end{equation}
	have specified fillers, and these fillers satisfy certain coherences in the form of associativity and partial naturality properties. Here, $\ell : \Delta^p \to \Delta^n$ is the left inclusion $d^n \cdots d^{p+1}$ and $r : \Delta^q \to \Delta^n$ is the right inclusion $d^0 \cdots d^0$, and similarly $\ell'$ and $r'$ are the left and right inclusions of the intersection simplex $\Delta^{p + q - n}$.
\end{defn}

Compositories describe simplicial sets in which some of the inner span inclusions we consider, namely the spans containing the first and last $(n-1)$-faces of the $n$-simplex, are assigned coherent choices of fillers. In many of the examples of compositories, however, these fillers are not unique and additional weak filler conditions are satisfied making them (inner) span complete simplicial sets. The structure of (inner) span completeness does not subsume that of compositories, but rather offers a complementary perspective on simplicial sets whose simplices can be combined to form higher dimensional simplices with varying  degrees of uniqueness.

\subsection*{Quasicategories}

Among the most basic classes of simplicial sets defined by filler conditions is the class of quasicategories. By definition, quasicategories have $n$-simplex fillers for all $n$-dimensional inner horns.

\begin{prop}
	\label{qcat_innerspans}
	Quasicategories are inner span complete.
\end{prop}

\begin{proof}
	Consider a basic inner span omitting respectively the $i$th and $j$th vertices of the $n$-simplex where $j - i > 1$.
	Choose $k$ to lie between $i$ and $j$.  Observe that this span contains precisely the faces of $\Delta^n$ not containing the edge from $i$ to $j$.  By \cite[Lemma 4.4.5.5]{lurie} applied with $J = \{ i, j \}$, the inclusion of the basic inner span into the $n$-simplex is inner anodyne, and therefore has fillers in any quasicategory.
\end{proof}

Moreover, it is straightforward to see that every Kan complex is span complete.

%The converse to this statement, that an inner span complete simplicial set is a quasicategory, is not true: below we show that the bar construction of an algebra of a BC monad is inner span complete, so that the unfilled inner horns of \Cref{unfillable_horns} provide counterexamples in bar constructions of BC monads.

\subsection*{Finite metrics}

\cite[Section 3.3]{FF} describes the simplicial set of finite metric spaces, which we now show to be span complete. Recall that a \emph{pseudometric} on a set $S$ is a function $d : S \times S \to \mathbb{R}_{\ge 0}$ such that
\begin{itemize}
	\item (Reflexivity) $d(x,x) = 0$ for all in $x \in S$,
	\item (Symmetry) $d(x,y) = d(y,x)$ for all $x,y \in S$,
	\item (Triangle inequality) $d(x,z) \le d(x,y) + d(y,z)$ for all $x,y,z \in S$.
\end{itemize}
Pseudometrics differ from metrics in that they do not require the non-degeneracy condition that $d(x,y) = 0$ only when $x=y$ in $S$. This weakening of the definition is necessary in order to define the degeneracy maps in the simplicial set constructed below. This construction would work just as well without imposing the symmetry condition on pseudometrics, so the thus inclined reader may as well drop the symmetry property. In the following, we will simply say \emph{metric} to refer to either type of pseudometric.

\begin{ex}
Let $\mathcal{M}_1(n)$ be the set of metrics on the set $[n]$ of $n+1$ points. For each map $f : [m] \to [n]$ in $\Delta$ and metric $d$ on $[n]$, define the metric $f^\ast d$ by $(f^\ast d)(x,y) = d(f(x),f(y))$ for $x,y \in [m]$.  When $f$ is a coface map, then $f^\ast$ takes a metric on $[n]$ and restricts it to the image of $[m]$ under $f$, and when $f$ is a codegeneracy map, it replaces each point $i$ in $[n]$ with the set $f^{-1}(i)$, all of which have zero distance from each other and such that they have the same distances to any other point in $[m]$.

A span in $\mathcal{M}_1$ amounts to a choice of metrics $d_p,d_q$ on $[p],[q]$ which agree upon restriction along the span's coface maps $[m] \to [p]$ and $[m] \to [q]$. A filler of this span is a metric $d_n$ on $[n] = [p+q-m] \cong [p] \cup_{[m]} [q]$ which restricts to $d_p,d_q$ on $[p],[q]$, respectively. A canonical (but not generally unique) choice of $d_n$ is given by
\[
	d_n(i,j) = \begin{cases}
		d_p(i,j) & \textrm{if } i,j \in [p], \\[3pt]
		d_q(i,j) & \textrm{if } i,j \in [q], \\[3pt]
		\minl\limits_{k \in [m]} \blp d_p(i,k) + d_q(k,j) \brp & \textrm{if } i \in [p]\backslash[m], \; j \in [q]\backslash[m],
	\end{cases}
\]
generalizing \cite[Definition 3.3.2]{FF} to any inclusions of $[m]$ into $[p]$ and $[q]$. Note that the first two cases are not disjoint, but result in the same value for $i,j \in [m]$ due to the assumption of equal restriction.

%Any filler of such a span must agree with this one on pairs of points either both in $[p]$ or both in $[q]$, and by the triangle inequality have distances less than or equal to these between a point only in $[p]$ and a point only in $[q]$.
Unlike the compository structure on $\mathcal{M}_1$, the property of span completeness does not single out this particular filler over other possible ones, but it does capture fillers of more spans (namely for any span of cofaces $[p] \leftarrow [m] \to [q]$ rather than just the one where $[m]$ is embedded as the initial or final face, respectively).

By \cite[Figure 5]{FF}, $\mathcal{M}_1$ is not a Kan complex. It follows that Kan complexes are a strict subclass of span complete simplicial sets. 
\end{ex}

\subsection*{Higher spans}

Following \cite[Section 3.2]{FF}, \emph{higher spans} describe sequences of adjacent spans in a category equipped with choices of the data necessary to ``compose'' them, in the following sense.

\begin{defn}
For $n \in \mathbb{N}$, we let $\Sp_n$ be the poset category pictured below.
\[\begin{tikzcd}[column sep={4.0em,between origins}, row sep={3.5em,between origins}]
	& & & & (0,n) \ar{dl} \ar{dr} & & & & \\
	& & & (0,n-1) \ar{dl} \ar{dr} & & (1,n) \ar{dl} \ar{dr} & & & \\
	& & \iddots \ar{dl} & & (1,n-1) \ar{dl} \ar{dr} & & \ddots \ar{dr} & & \\
	& (0,1) \ar{dl} \ar{dr} & & \iddots \ar{dl} & & \ddots \ar{dr} & & (n-1,n) \ar{dl} \ar{dr} & \\
	(0,0) & & (1,1) & & \cdots & & (n-1,n-1) & & (n,n)
\end{tikzcd}\]
\end{defn}

$\Sp$ forms a cosimplicial object in $\Cat$ where the coface map $d^i : \Sp_{n-1} \to \Sp_n$ acts componentwise via the usual
\[
	j \longmapsto
	\begin{cases}
		j+1 & \textrm{if } j \ge i, \\
		j   & \textrm{if } j < i,
	\end{cases}
\]
and similarly the codegeneracy $s^i : \Sp_{n+1} \to \Sp_n$ acts componentwise via the usual
\[
	j \longmapsto
	\begin{cases}
		j-1 & \textrm{if } j > i, \\
		j   & \textrm{if } j \le i.
	\end{cases}
\]
Equivalently, $\Sp : \Delta \to \Cat$ is the composite of the inclusion functor $\Delta \to \Cat$ with the twisted arrow category functor $\Cat \to \Cat$.

\begin{defn}
An $n$-span in a category $\C$ is a functor $\Sp_n \to \C$.
\end{defn}

Conceptually, an $n$-span describes a sequence of $n$ adjacent spans and coherent choices of ``composite'' spans for each connected subsequence of spans.
%We denote the resulting functor category by $\cat{Fun}(\Sp_n, \Cat)$.

\begin{ex}
	For a category $\C$, let $\mathcal{S}_\C$ be the simplicial set with $n$-simplices the set of $n$-spans $\Sp \to \C$, with simplicial structure maps induced by the cosimplicial structure of $\Sp$.

	This means that the $i$th face map $d_i$ forgets all objects of the span with an $i$ in either component, and then composes the remaining maps as needed, while the $i$th degeneracy map $s_i$ repeats the $i$th row and column with identities between them.

	An inner span (in our usual sense) in the simplicial set $\mathcal{S}_\C$ consists of a $p$-span and a $q$-span sharing an $m$-span as a face, and together containing a sequence of $n$ adjacent $1$-spans between their bottom objects. Below are two schematic examples of basic inner spans in $\mathcal{S}_\C$. 
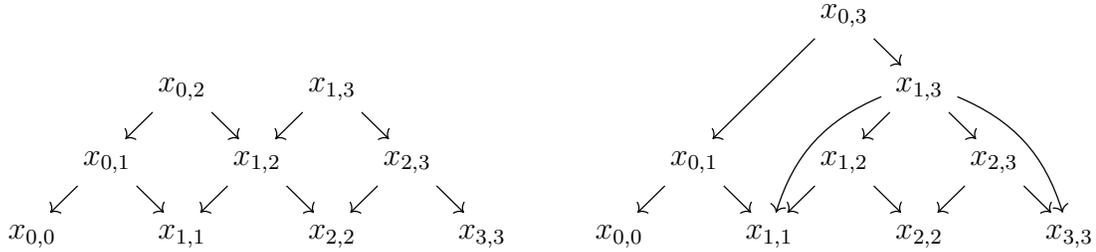
\[\hspace{-0cm}\begin{tikzcd}[column sep=0, row sep=small]
	& & & \textcolor{white}{x_{0,3}} \\
	& & x_{0,2} \ar{dl} \ar{dr} & & x_{1,3} \ar{dl} \ar{dr} \\
	& x_{0,1} \ar{dl} \ar{dr} & & x_{1,2} \ar{dl} \ar{dr} & & x_{2,3} \ar{dl} \ar{dr} \\
	x_{0,0} & & x_{1,1} & & x_{2,2} & & x_{3,3}
\end{tikzcd}\qquad \begin{tikzcd}[column sep=0, row sep=small]
	& & & x_{0,3} \ar{ddll} \ar{dr} \\
	& &  & & x_{1,3} \ar{dl} \ar{dr} \ar[bend right=25]{ddll} \ar[bend left=25]{ddrr} \\
	& x_{0,1} \ar{dl} \ar{dr} & & x_{1,2} \ar{dl} \ar{dr} & & x_{2,3} \ar{dl} \ar{dr} \\
	x_{0,0} & & x_{1,1} & & x_{2,2} & & x_{3,3}
\end{tikzcd}\]
In the basic inner span above left, consisting of the first and last faces of a $3$-span, a $3$-span filler requires the additional data of the top span making the resulting square commute. This will always exist (albeit non-uniquely) if $\C$ has pullbacks (or more generally if $\C$ is cofiltered, but this will not be enough to fill the other inner spans).

The example above right shows an inner span consisting of the faces $\{0,1,3\}$ and $\{1,2,3\}$ in the $3$-simplex. These two together contain all of the data of a $3$-span except for the object $x_{0,2}$ and its maps
\[
	x_{0,3} \to x_{0,2}, \qquad x_{0,2} \to x_{0,1}, \qquad x_{0,2} \to x_{1,2}.
\]
These can be filled by taking $x_{0,2}$ to be the pullback of $x_{0,1}$ and $x_{1,2}$, with the map $x_{0,3} \to x_{0,2}$ induced by the universal property of the pullback, which also shows that the resulting upper square commutes. These fillers cannot be expected to be unique, however, as for instance this inner span has a different filler given by setting $x_{0,2} := x_{0,3}$.

More generally, a basic inner span consisting of the $i$th and $j$th faces of a potential $n$-span, with $j - i > 1$, contains all the data of an $n$-span except for the object $x_{i,j}$ and the four maps into and out of it, as below. 
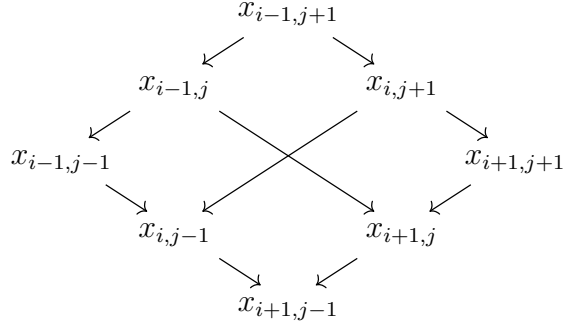
\[\begin{tikzcd}[column sep=0, row sep=small]
	& & & x_{i-1,j+1} \ar{dl} \ar{dr} \\
	& & x_{i-1,j} \ar{dl} \ar{ddrr} & & x_{i,j+1} \ar{ddll} \ar{dr} \\
	& x_{i-1,j-1} \ar{dr} & &  & & x_{i+1,j+1} \ar{dl} \\
	& & x_{i,j-1} \ar{dr} & & x_{i+1,j} \ar{dl} \\
	& & & x_{i+1,j-1}
\end{tikzcd}\]
To fill this data to an entire $n$-span, we must specify an object $x_{i,j}$ and four maps filling the diagram above into a 2 by 2 commuting grid. Generalizing both of the previous examples, we can take $x_{i,j}$ to be a pullback $x_{i,j-1} \times_{x_{i+1,j-1}} x_{i+1,j}$, the two maps out of $x_{i,j}$ the canonical projections to $x_{i,j-1}$ and $x_{i+1,j}$, and the two maps into $x_{i,j}$ those induced by 
\[
	x_{i,j-1} \leftarrow x_{i-1,j-1} \leftarrow x_{i-1,j} \to x_{i+1,j} \quad\: \textrm{and} \quad\: x_{i,j-1} \leftarrow x_{i,j-1} \to x_{i+1,j+1} \to x_{i+1,j}
\]
respectively using the universal property of the pullback, which also ensures that the left and right squares of the grid commute. The universal property also guarantees that the top square commutes, as either side satisfies the defining property of the unique map $x_{i-1,j+1} \to x_{i,j}$ induced by $x_{i,j-1} \leftarrow x_{i-1,j+1} \to x_{i+1,j}$. 

This shows every basic inner span has a filler to an $n$-simplex, so if $\C$ has pullbacks $\mathcal{S}_\C$ is inner span complete by \Cref{complete_characterization}. Note that $\mathcal{S}_\C$ is not generally span complete even when $\C$ has pullbacks, as in the above if $j = i+1$ the object $x_{i,i+1}$ cannot be recovered by a pullback, though it can be recovered in a similar fashion as $x_{i,i} \times x_{i+1,i+1}$ if $\C$ has products, in which case $\mathcal{S}_\C$ is span complete again by \Cref{complete_characterization}.

In \cite[Section 3.2]{FF}, $\C$ is further assumed to be gaunt\footnote{A category is gaunt if the only isomorphisms are the identities.}, so that pullbacks are strictly unique. This implies the relevant coherences for the resulting fillers, and the difficulties of making $\mathcal{S}_\C$ into a compository when pullbacks are not strictly unique are also sketched. Inner span completeness provides a different description of the structure of $\mathcal{S}_\C$ which allows for such non-uniqueness of inner span fillers, but in return does not describe the up-to-isomorphism coherence properties of pullbacks that the compository structure captures when the pullbacks are unique. A more complete description of this particular structure on $\mathcal{S}_\C$ when $\C$ has pullbacks remains open.

Moreover,~\cite[Example 3.2.5]{FF} shows that when $\C$ is the category with a single commuting square of morphisms, $\mathcal{S}_\C$ is not a quasicategory, providing another example of how span filling properties are strictly weaker than horn filling.
\end{ex}

\subsection*{Gleaves on $\cat{FinSet}$}

Many of the known examples of compositories such as the simplicial set of finite metric spaces above
---with the higher span example as a notable exception---are in fact (augmented) \emph{symmetric} simplicial sets in a natural way, or equivalently \emph{gleaves} on $\cat{FinSet}$~\cite[Section~5]{FF}. While we will not need the precise definition here, we only note that every gleaf on $\cat{FinSet}$ is in particular a span complete simplicial set: the filler condition for basic coface squares~\eqref{basic_coface_squares} holds for $i = 0$ and $j = n$ as part of the algebraic structure carried by a gleaf (as it already does for compositories), and this is enough to prove the filler condition in general by symmetry.

Our final two examples are particular instances of this general construction of span complete simplicial sets from gleaves.

\begin{ex}[Joint probability distributions]
	\label{joint_distributions}
	We now sketch the main example of~\cite{FF} coming from probability theory. Fix a finite set $S$, representing the set of possible values of some random variables; the finiteness assumption on $S$ is merely for technical simplicity. Then define an $n$-simplex to be a joint distribution of $(n+1)$ many $S$-valued random variables, i.e.~a probability measure on the cartesian product $S^{\times (n+1)}$. Taking the pushforward of this measure along a projection to a subproduct is then what defines the faces of such a simplex, and similarly taking the pushforward along the diagonal $S \to S \times S$ is involved in the definition of the degeneracy maps, which intuitively amounts to duplicating the value of a variable. We thus obtain the desired simplicial set in which the $n$-simplices are the joint distributions of $n+1$ random variables. The fillers of~\Cref{compository_filler} then acquire the following significance. Suppose that we are given a set of $n+1$ random variables and write it as the union of subsets containing $p+1$ and $q+1$ random variables, respectively. Then if we are given a joint distribution of the $p+1$ variables and a joint distribution of the $q+1$ variables, and if these two agree when marginalized to the subsubset of variables contained in both subsets, then there is joint distribution for all $n+1$ variables which marginalizes to the given ones. And indeed there is a distinguished choice for this overall joint distribution: we can make all variables contained in the first subset but not the second to be conditionally independent of those in the second set but not the first, conditionally with respect to the variables contained in both, which is exactly the conditional product that we also used in \cite[Example 2.8]{weakly_cartesian}. 
	In particular, this simplicial set is span complete.

	However, fillers for (inner or outer) $3$-horns do not generally exist as soon as $|S| \ge 2$. For then, we can consider without loss of generality $S = \{0,1\}$, and four random variables $A,B,C,D$ where $D = 0$ with probability $1$, and the other three such that they take either value with probability $1/2$, but are correlated such that they take \emph{opposite} values with probability $1$. This defines joint distributions of $ABD$, $ACD$ and $BCD$, thereby forming a $3$-horn in the corresponding symmetric simplicial set. Similar to the examples in the proof of \cite[Theorem 4.7]{weakly_cartesian}, 
	this $3$-horn is such that already its missing $2$-face cannot be filled: there is no joint distribution of $ABC$ which would make all three variables take opposite values with probability $1$, since there is not even a single assignment of values in $S$ which would assign opposite values to every pair.

	We also refer to~\cite[Section~12]{markov_cats}, where a more general construction of gleaves of this type has been proposed. This in particular applies in the context of infinite $S$ and measure-theoretic probability.
\end{ex}

\begin{ex}[Relational databases]
	\label{databases}
	A closely related example comes from the theory of relational databases~\cite{maier}. Again fixing a set $S$ for possible values, which now plays the role of the set of possible values for the entries in a table of a database, in this case an $n$-simplex is defined to be a \emph{subset} of $S^{\times (n+1)}$ (rather than a probability measure as in the previous table). The idea is to interpret the individual factors of $S^{\times (n+1)}$ as the columns in a table of a database, so that the subset is the set of rows that appear in the table.

	These simplices assemble into a simplicial set, where the face maps are given by projecting a subset to a subproduct (deleting one or more columns in a table and eliminating duplicates), and the generating degeneracy maps are defined again in terms of pushforward along the diagonal $S \to S \times S$ (duplicating a column). More abstractly, this simplicial set can be described as the composite functor
\[
	\begin{tikzcd}
		\Delta^\op \ar{r} & \cat{FinSet}^\op \ar{r}{S^{(-)}} & \Set \ar{r}{2^{(-)}} & \Set
	\end{tikzcd}
\]
Moreover, it is in a canonical way a symmetric simplicial set with respect to permutation of the factors.

	This simplicial set has properties closely analogous to the previous example. In particular it has fillers for all inner spans, which we now describe with the database terminology. We thus assume that $A \subseteq S^{p+1}$ and $B \subseteq S^{q+1}$ are tables in a database with $p+1$ and $q+1$ columns respectively, and that they have $m+1$ columns in common, so that dropping the other columns from either table results in the same $(m+1)$-column table (after removing duplicates). Then there is a \emph{maximal} way to create a table with $n + 1$ columns, where $n = p + q - m$, given by using all conceivable rows whose restriction to the $p+1$ attributes of $A$ occurs in $A$ and whose restriction to the $q+1$ attributes of $B$ occurs in $B$. This is known as the \emph{join} of $A$ and $B$~\cite[Section~2.4]{maier}.

	Therefore the simplicial set is indeed span complete. In particular, we also obtain fillers for all acyclic configurations by~\Cref{acyclic_fillers}. This is a classic result of relational database theory~\cite{maier}\footnote{See property 5 of Theorem 13.2 in there.}.
\end{ex}

\bibliographystyle{plain} % was: alpha

\begin{thebibliography}{10}

\bibitem{acyclic_desirability}
Catriel Beeri, Ronald Fagin, David Meier, and Mihalis Yannakakis.
\newblock On the desirability of acyclic database schemes.
\newblock {\em Journal of the Association for Computing Machinery},
  30(3):479--513, 1983.

\bibitem{weakly_cartesian}
Carmen Constantin, Tobias Fritz, Paolo Perrone, and Brandon Shapiro.
\newblock Partial evaluations and the compositional structure of the bar
  construction, 2020.
\newblock \href{https://arxiv.org/abs/2009.07302}{arXiv:2009.07302}.

\bibitem{DK2segal}
Tobias Dyckerhoff and Mikhail Kapranov.
\newblock {\em Higher {S}egal spaces}, volume 2244 of {\em Lecture Notes in
  Mathematics}.
\newblock Springer, Cham, 2019.
\newblock \href{https://arxiv.org/abs/1212.3563}{arXiv:1212.3563}.

\bibitem{2segalunital}
Matthew Feller, Richard Garner, Joachim Kock, May~U. Proulx, and Mark Weber.
\newblock Every 2-{S}egal space is unital.
\newblock {\em Commun. Contemp. Math.}, 23(2):2050055, 6, 2021.
\newblock \href{https://arxiv.org/abs/1905.09580}{arXiv:1905.09580}.

\bibitem{FF}
Cecilia Flori and Tobias Fritz.
\newblock Compositories and gleaves.
\newblock {\em Theory Appl. Categ.}, 31(33):928--988, 2016.
\newblock \href{https://arxiv.org/abs/1308.6548}{arXiv:1308.6548}.

\bibitem{markov_cats}
Tobias Fritz.
\newblock A synthetic approach to {M}arkov kernels, conditional independence
  and theorems on sufficient statistics.
\newblock {\em Adv. Math.}, 370:107239, 2020.
\newblock \href{https://arxiv.org/abs/1908.07021}{arXiv:1908.07021}.

\bibitem{FP}
Tobias Fritz and Paolo Perrone.
\newblock Monads, partial evaluations, and rewriting.
\newblock {\em Proceedings of MFPS 36, ENTCS}, 2020.
\newblock \href{https://arxiv.org/abs/1810.06037}{arXiv:1810.06037}.

\bibitem{GaleCyclic}
David Gale.
\newblock Neighborly and cyclic polytopes.
\newblock In {\em Proc. Sympos. Pure Math}, volume~7, pages 225--232, 1963.

\bibitem{GKT1}
Imma G{\'a}lvez-Carrillo, Joachim Kock, and Andrew Tonks.
\newblock Decomposition spaces, incidence algebras and {M}{\"o}bius inversion
  {I}: Basic theory.
\newblock {\em Advances in Mathematics}, 331:952--1015, 2018.
\newblock
  \href{http://dx.doi.org/10.1016/j.aim.2018.03.016}{dx.doi.org/10.1016/j.aim.2018.03.016}.

\bibitem{GKT2}
Imma G{\'a}lvez-Carrillo, Joachim Kock, and Andrew Tonks.
\newblock Decomposition spaces, incidence algebras and {M}{\"o}bius inversion
  {II}: Completeness, length filtration, and finiteness.
\newblock {\em Advances in Mathematics}, 333:1242--1292, 2018.
\newblock
  \href{http://dx.doi.org/10.1016/j.aim.2018.03.017}{dx.doi.org/10.1016/j.aim.2018.03.017}.

\bibitem{graham}
M.~H. Graham.
\newblock On the universal relation, 1979.
\newblock Technical Report, University of Toronto.

\bibitem{JoyalAnalytic}
Andr{\'e} Joyal.
\newblock Foncteurs analytiques et esp{\`e}ces de structures.
\newblock In Gilbert Labelle and Pierre Leroux, editors, {\em Combinatoire
  {\'e}num{\'e}rative}, volume 1234 of {\em Lecture Notes in Math.}, pages
  126--159. Springer Berlin Heidelberg, 1986.

\bibitem{joyal}
Andr\'e Joyal.
\newblock Quasi-categories and {K}an complexes.
\newblock {\em J. Pure Appl. Algebra}, 175(1-3):207--222, 2002.
\newblock Special volume celebrating the 70th birthday of Professor Max Kelly.

\bibitem{JoyalTierneyNotes}
Andr\'e Joyal and Miles Tierney.
\newblock Notes on simplicial homotopy theory, 2008.
\newblock
  \href{http://mat.uab.cat/~kock/crm/hocat/advanced-course/Quadern47.pdf}{http://mat.uab.cat/$\sim$kock/crm/hocat/advanced-course/Quadern47.pdf}.

\bibitem{lurie}
Jacob Lurie.
\newblock {\em Higher Topos Theory}.
\newblock Annals of Mathematics Studies. Princeton University Press, 2009.
\newblock \href{https://arxiv.org/abs/math/0608040}{arXiv:math/0608040}.

\bibitem{maier}
David Maier.
\newblock {\em The theory of relational databases}.
\newblock Computer Software Engineering Series. Computer Science Press,
  Rockville, MD, 1983.

\bibitem{nikolaus}
Thomas Nikolaus.
\newblock Algebraic models for higher categories.
\newblock {\em Indag. Math. (N.S.)}, 21(1-2):52--75, 2011.
\newblock \href{https://arxiv.org/abs/1003.1342}{arXiv:1003.1342}.

\bibitem{Poguntke}
Thomas Poguntke.
\newblock Higher {S}egal structures in algebraic {$K$}-theory, 2017.
\newblock \href{https://arxiv.org/abs/1709.06510}{arXiv:1709.06510}.

\bibitem{catht}
Emily Riehl.
\newblock {\em Categorical Homotopy Theory}.
\newblock Cambridge University Press, 2013.

\bibitem{Walde}
Tashi Walde.
\newblock Higher {S}egal spaces via higher excision, 2019.
\newblock \href{https://arxiv.org/pdf/1906.10619}{arXiv:1906.10619}.

\bibitem{wegner}
Gerd Wegner.
\newblock {$d$}-collapsing and nerves of families of convex sets.
\newblock {\em Arch. Math. (Basel)}, 26:317--321, 1975.

\bibitem{west}
Douglas~B. West.
\newblock {\em Introduction to graph theory}.
\newblock Prentice Hall, Inc., Upper Saddle River, NJ, 1996.

\end{thebibliography}

\end{document}